\documentclass{article}

\usepackage{makecell}
\usepackage{tikz}
\usepackage{graphicx}
\usepackage{enumitem}
\usepackage{float}
\usetikzlibrary{angles, calc}
\usetikzlibrary{arrows.meta}
\usepackage{extarrows}
\usetikzlibrary{arrows}
\usepackage{pgfplots}
\usepackage{pgfplotstable}
\pgfplotsset{width=7cm,compat=1.18}
\usetikzlibrary{decorations.pathreplacing,decorations.markings}
 \tikzset{
  on each segment/.style={
    decorate,
    decoration={
      show path construction,
      moveto code={},
      lineto code={
        \path [#1]
        (\tikzinputsegmentfirst) -- (\tikzinputsegmentlast);
      },
      curveto code={
        \path [#1] (\tikzinputsegmentfirst)
        .. controls
        (\tikzinputsegmentsupporta) and (\tikzinputsegmentsupportb)
        ..
        (\tikzinputsegmentlast);
      },
      closepath code={
        \path [#1]
        (\tikzinputsegmentfirst) -- (\tikzinputsegmentlast);
      },
    },
  },
  mid arrow/.style={postaction={decorate,decoration={
        markings,
        mark=at position .5 with {\arrow[#1]{stealth}}
      }}},
}
\usetikzlibrary{arrows}

\usepackage[pagewise]{lineno}
\RequirePackage{amsthm,amsmath,amsfonts,amssymb,fontenc,diagbox,listings,verbatim}
\numberwithin{equation}{section}
\RequirePackage[numbers]{natbib}
\RequirePackage[colorlinks,citecolor=blue,urlcolor=blue]{hyperref}
\RequirePackage{graphicx}
\RequirePackage[english]{babel}
\RequirePackage[T1]{fontenc}
\usepackage{extarrows}
\usepackage{bbm}
\theoremstyle{plain}

\newtheorem{theorem}{Theorem}[section]
\newtheorem{lemma}[theorem]{Lemma}
\newtheorem{corollary}[theorem]{Corollary}
\newtheorem{remark}{Remark}[section]

\theoremstyle{remark}

\newtheorem*{example}{Example}

\usepackage{PRIMEarxiv}

\usepackage[utf8]{inputenc} % allow utf-8 input
\usepackage[T1]{fontenc}    % use 8-bit T1 fonts
\usepackage{hyperref}       % hyperlinks
\usepackage{url}            % simple URL typesetting
\usepackage{booktabs}       % professional-quality tables
\usepackage{amsfonts}       % blackboard math symbols
\usepackage{nicefrac}       % compact symbols for 1/2, etc.
\usepackage{microtype}      % microtypography
\usepackage{lipsum}
\usepackage{fancyhdr}       % header
\usepackage{graphicx}       % graphics
\graphicspath{{media/}}     % organize your images and other figures under media/ folder

\title{Three-Parameter Approximations of  Sums of Locally Dependent Random Variables via Stein's Method}

\author{
  Zhonggen Su, Xiaolin Wang \\
  School of Mathematical Sciences, Zhejiang University \\
  Hangzhou{\rm310058}, China\\
\texttt{\{suzhonggen, 12035020\}@zju.edu.cn}}

\begin{document}
\maketitle
\begin{abstract}
Let $\{X_{i}, i\in J\}$ be a  family of locally dependent non-negative integer-valued random variables with finite expectations and variances. We consider the sum $W=\sum_{i\in J}X_i$ and use Stein's method to establish general upper error bounds for the total variation distance $d_{TV}(W, M)$, where $M$ represents a three-parameter random variable. As a direct consequence, we obtain a  discretized normal approximation for $W$. As applications, we study in detail four well-known examples, which are counting vertices of all edges point inward, birthday problem, counting monochromatic edges in uniformly colored graphs, and triangles in the Erd\H{o}s-R\'{e}nyi random graph. Through delicate analysis and computations, we obtain sharper upper error bounds than existing results.
\end{abstract}
\section{Introduction}\label{S2}
Locally dependent random variables appear widely in many applied fields, such as risk theory, extreme value theory, reliability theory, run and scan statistics, and graph theory. For related background, the interested reader is referred to \cite{barbour1992poisson,goovaerts1996compound,rollin2008symmetric, smith1988extreme,wainer1996reliability}. The purpose of this paper is to provide three-parameter approximations for sums of locally dependent random variables and to closely examine a few well-known examples.

Let  $\{X_{i}, i\in J\}$ be a family of discrete non-negative integer-valued random variables, where $J$ is an index set. Assume that $\{X_{i}, i\in J\}$
satisfies the following local dependence structure:

 (LD1)  For each $i \in J$, there exists an set $A_{i} \subset J$ such that $X_{i}$ is independent of $\{X_{j} : j \notin A_{i}\}$.

 (LD2)  For each $i \in J$ and $j \in A_{i},$ there exists an set $A_{ij} \supset A_{i}$ such that $\{X_{i}, X_{j}\}$ is independent of $\{X_{k} : k \notin A_{ij}\}$.

 (LD3)   For each $i \in J$, $j \in A_{i}$, and $k \in A_{ij}$, there exists an set $A_{ijk} \supset A_{ij}$ such that $\{X_{i}, X_{j}, X_{k}\}$ is independent of $\{X_{l} : l \notin A_{ijk}\}$.

 \noindent
We remark that such a  dependence structure follows  Ross \cite{ross2011fundamentals} and was also adopted by  Fang \cite{fang2019wasserstein}  and Liu  and Austern \cite{liu2022wasserstein}. Let $W=\sum_{i\in J}X_{i}$, and assume  that ${\rm E}W^2$ exists. Set $\mu={\rm E}W$ and $\sigma^{2}=\operatorname{Var}W$.
Our aim is to present a general approach for obtaining an effective approximation of $W$ by appropriate discrete random variables with three parameters, in terms of the total variation distance and the local distance. Recall that the total variation distance and the local distance between two integer-valued random variables $U$ and $V$ are defined as
 $$d_\mathrm{T V}(U, V)=\sup _{A \subseteq \mathbb{Z}}|{\rm P}(U \in A)-{\rm P}(V \in A)| $$
 and
$$d_{\mathrm{loc}}(U, V)=\sup _{a \in \mathbb{Z}}|{\rm P}(U=a)-{\rm P}(V=a)|.$$

Let $F$ and $G$ be two probability distribution functions, and denote by $F*G$ their convolution, namely, $F*G(z)=\int_{\mathbb{R}} F(z-x)dG(x)$. The major task of this paper is to  control the approximation error bounds $d_\mathrm{TV}(W, M_i)$, $i=1, 2, 3$, where    random variables  $M_i$  with three parameters are defined as follows:
\begin{eqnarray}
M_{1}\sim B(n,p)*\mathcal{P}(\lambda), \quad M_{2}\sim NB(r,p)* \mathcal{P}(\lambda), \quad M_{3}\sim \mathcal{P}(\lambda)* 2\mathcal{P}\bigg(\frac \omega 2\bigg)* 3\mathcal{P}\bigg(\frac \eta 3\bigg).\label{MMMM}
\end{eqnarray}
Here and in the sequel,  $X\sim F$ means that $X$ is a random variable with distribution $F$.

A key step is to properly select the parameters in (\ref{MMMM}). Our guiding principle is to ensure that the first three factorial cumulants of $M_{i}$ are equal (or approximately equal) to those of $W$. This allows us to effectively control $d_{\mathrm{TV}}(W, M_{i})$ as desired.

Specifically, given a  non-negative integer-valued random variable $U$, let $\phi_{U}(z)$ denote its probability generating function (PGF), i.e., $\phi_{U}(z)={\rm E}z^U$, and  the $j$-th factorial cumulant  $\Gamma_{j}(U)$  is defined as
\begin{equation*}\label{defff}
  \Gamma_{j}(U)=\frac{d^{j}}{d z^{j}} \log \phi_{U}(z)\bigg|_{z=1}.
 \end{equation*}
Let $m_{i}$ denote   the $i$-th  moment  of $W$. Then,  it is immediate that
\begin{eqnarray*}\label{wkk1}
 \Gamma_{1}(W)=m_{1},\quad \Gamma_{2}(W)= m_{2}-m_{1}^{2}-m_{1},
\end{eqnarray*}
and
\begin{eqnarray*}\label{wkk2}
 \Gamma_{3}(W)=m_{3}-3m_{1}m_{2}-3m_{2}+2m_{1}^{3}+3m_{1}^{2}+2m_{1}.
\end{eqnarray*}
Furthermore, some simple calculations yield
\begin{align}
    & \Gamma_{1}(M_{1})=np+\lambda, \quad\Gamma_{2}(M_{1})=-np^{2},\quad\Gamma_{3}(M_{1})=2np^{3},\label{k1}\\
    & \Gamma_{1}(M_{2})=\frac{rq}{p}+\lambda,\quad
    \Gamma_{2}(M_{2})=\frac{rq^{2}}{p^{2}},\quad
    \Gamma_{3}(M_{2})=2\frac{rq^{3}}{p^{3}}, \label{k2}\\
    & \Gamma_{1}(M_{3})=\lambda+\omega+\eta, \quad\Gamma_{2}(M_{3})=\omega+2\eta,\quad \Gamma_{3}(M_{3})=2\eta.\label{k3}
\end{align}

Now, consider the following system  of equations:
\begin{eqnarray}\label{G1t}
\left\{\!
\begin{array}{l}
    \Gamma_{1}(W)=\Gamma_{1}(M_{i}),\\
     \Gamma_{2}(W)=\Gamma_{2}(M_{i}),\\
   \Gamma_{3}(W)=\Gamma_{3}(M_{i}).
   \end{array}
    \right.
\end{eqnarray}
  Using (\ref{k1})--(\ref{k3}) and solving   (\ref{G1t}),  simple but lengthy calculations yield
\begin{align}\label{K1}
     &n=  \bigg\lfloor\frac{{4\Gamma_{2}(W)}^{3}}{{\Gamma_{3}(W)}^{2}}\bigg\rfloor, \quad p=-\frac{\Gamma_{3}(W)}{2\Gamma_{2}(W)}, \quad q=1-p,\quad \lambda=\mu-np, \quad \delta=\frac{{4\Gamma_{2}(W)}^{3}}{{\Gamma_{3}(W)}^{2}}-n,\\\label{K2}
     &
     r=\frac{4\Gamma_{2}(W)^{3}}{\Gamma_{3}(W)^{2}},  \quad p = \frac{2\Gamma_{2}(W)}{2\Gamma_{2}(W) +\Gamma_{3}(W)}, \quad q=1-p,\quad \lambda =\mu - \frac{rq}{p},\\\label{K3}
     &  \lambda=\Gamma_{1}(W)-\Gamma_{2}(W)+\frac{1}{2}\Gamma_{3}(W), \quad \omega=\Gamma_{2}(W)-\Gamma_{3}(W), \quad \eta=\frac{1}{2}\Gamma_{3}(W).
\end{align}
Note that $n$ in (\ref{K1}) may not be an integer, where we   take its integer part.

Another technical issue is:  which $M_i$ is selected to approximate the target variable $W$? It turns out to depend on the size of $\Gamma_{2}(W)$. Specifically, we consider three cases separately: (1) select $M_1$ if $\Gamma_{2}(W)> 0$;  (2) select $M_2$ if $\Gamma_{2}(W)< 0$; (3)  select $M_{3}$ if $\Gamma_{2}(W)/\mu\approx 0$. See Remark \ref{NR1} below for more details. We are now ready to state our main results. Some additional notations are needed. Set $ X_{A_{ijk}}=\{X_{l},l\in A_{ijk}\}$, and define
\begin{align*}
 &S_{2}(W)=\sum_{k=0}^{\infty} | {\rm P}(W=k)-2{\rm P}(W=k-1)+{\rm P}(W=k-2) |,  \notag\\
&
 S_{2}(W|X_{A_{ijk}})=\sum_{k=0}^{\infty} | {\rm P}(W=k|X_{A_{ijk}})-2{\rm P}(W=k-1|X_{A_{ijk}})+{\rm P}(W=k-2|X_{A_{ijk}}) |,  \notag\end{align*}
 and
\begin{align}
 S_{i,j,k}(W) =\operatorname{esssup}S_{2}(W|X_{A_{ijk}}),\quad
 S(W)=\sup _{i,j,k}S_{i,j,k}(W),
\end{align}
where $\operatorname{esssup}X$ denotes the essential supremum of $X$.

Also, define
\begin{align*}
\gamma_{i}&=\sum_{j \in A_{i}\backslash\{i\}}  \sum_{k \in A_{ij}}\sum_{l \in A_{ijk}}  \sum_{({\rm E})}{\rm E}X_{i}({\rm E})X_{j}{\rm E}X_{k}({\rm E})X_{l}\nonumber\\
     &\quad\,+\sum_{j \in A_{i}\backslash\{i\}}  \sum_{k \in A_{ij}\backslash A_{i}}\sum_{l \in A_{ijk}}  \sum_{({\rm E})}{\rm E}X_{i}({\rm E})X_{j}X_{k}({\rm E})X_{l},\nonumber\\
      \gamma&=\sum_{i\in J}\gamma_{i},
 \end{align*}
where $({\rm E})X_{i}$ stands for $X_{i}$ or ${\rm E}X_{i}$, and $\sum _{({\rm E})}$ denotes the sum over all possible choices of ${\rm E}$ in front of each $X_{i}$. For example,
 \begin{align*}
    \sum_{({\rm E})}{\rm E}X_{i}({\rm E})X_{j}({\rm E})X_{k}={\rm E}X_{i}X_{j}X_{k}+{\rm E}X_{i}X_{j}{\rm E}X_{k}+{\rm E}X_{i}{\rm E}X_{j}X_{k}+{\rm E}X_{i}{\rm E}X_{j}{\rm E}X_{k}.
\end{align*}

 \begin{theorem}\label{main} Use the above setting and notations.

{\rm(i)} Assume $\mu > \sigma^{2}$ and $2 \lambda  < q\lfloor n+ \lambda/p\rfloor$, where $\{n,p,q,\lambda\}$ is given in  \eqref{K1}. Then,  we have
$$d_\mathrm{TV}(W,M_{1})\leq C\frac{\gamma S(W)+p^2}{(1-2\theta_{1})q\mu},$$
where $\theta_{1} =\lambda/(\lfloor n+\lambda/p\rfloor q)$.

{\rm(ii)} Assume $\mu < \sigma^{2}$ and $2 \lambda q < rq+\lambda p $, where $\{r,p,q,\lambda\}$ is given in \eqref{K2}.
 Then,  we have
$$d_\mathrm{TV}(W,M_{2})\leq  C \bigg(1\vee \frac{q}{p}\bigg )\frac{\gamma S(W)}{(1-2\theta_{2})\mu},$$
where $\theta_{2} =\lambda q/(rq+\lambda p)$.

{\rm(iii)} Assume $\mu/\sigma^{2} \approx 1$ and $2 \omega+4\eta < \lambda+\omega+\eta$, where $\{\lambda, \omega, \eta\}$ is given in \eqref{K3}. Then,  we have
$$d_\mathrm{TV}(W,M_{3})\leq C\frac{\gamma S(W)}{(1-2\theta_{3})\mu},$$
where $\theta_{3} = (\omega+2\eta)/(\lambda+\omega+\eta)$.

Here and in the sequel, $C$ stands for a numerical constant whose value may differ from line to line.
\end{theorem}
\begin{remark}\label{NR1}
We adapt different three-parameter random variables to approximate $W$ according to $\mu$ and $\sigma^2$. In principle, we prefer to use $M_1$ (resp. $M_2$) if $\mu$ is significantly greater (resp. less) than $\sigma^2$, while we would rather use $M_3$ if $\mu$ is of roughly the same magnitude as $\sigma^2$; see Section 3 for concrete examples. There is no strict rule regarding variable selection. Similar issues have arisen in the study of two-parameter approximations. For example, Barbour and Xia \cite{barbour1999poisson} used $CP(\lambda, \omega)$ to approximate the number of $2$-runs, while Brown and Xia \cite{brown2001stein} used $NB(r, p)$, both achieving $O(\sigma^{-1})$ error bounds.
\end{remark}

\begin{remark}\label{r1.3}
To apply Theorem \ref{main}, we usually need to efficiently estimate $S(W)$. In many real problems, this is equivalent to estimating $S_{2}(W)$. Here,  we  review the idea introduced in
\cite[Subsection 5.4]{rollin2008symmetric}.  Suppose there exists a random variable $Z$ defined on the same probability space as $W$, such that  $W$ conditioned on $Z=z$ can be represented as a sum of independent summands, i.e., $\mathcal{L}(W| Z=z)=\mathcal{L}(\sum_{i=1}^{n_{z}}X_{i}^{z})$, where $X_{i}^{z}$'s are independent for
each $z$ that $Z$ can attain. In turn, for   $U_z:= \sum_{i=1}^{n_{z}}X_{i}^{z}$, Barbour and Xia~\cite[Proposition 4.6]{barbour1999poisson} showed that
$$S_{2}(U_z)\leq \frac{8}{V_z-2v_z^{\star}},$$
where $V_{z}$ and $v_{z}^{\star}$ are  defined  for each $z$ as follows:
\begin{align}
&\label{2310}
 S_{1}(X_{i}^{z})=\sum_{k\in \mathbb{Z}}|{\rm P}(X_{i}^{z}=k+1)-{\rm P}(X_{i}^{z}=k)|,\\
&\label{2311}
    v_{i} = \min\{0.5, 1-0.5S_{1}(X_{i}^{z})\},\quad V_z =\sum_{i=1}^{n}v_{i}, \quad v_z^{\star}=\max_{i}v_{i}.
\end{align}
 Thus, we obtain
\begin{align}\label{2311}
   S_{2}(W)=  {\rm E}[S_{2}(W | Z)]={\rm E}[S_{2}(U_{ Z})]\leq  {\rm E}\frac{8}{V_{Z}-2v_{Z}^{\star}}.
\end{align}
\end{remark}
In addition to the above integer-valued approximation, it is natural to ask whether there is a good normal approximation for $W$.  Denote by $Y^{d}(\mu, \sigma^{2})$ the discretized normal random variable, i.e.,
$${\rm P} (Y^{d}(\mu, \sigma^{2})=k )=
 \int_{k-0.5}^{k+0.5}\frac{1}{\sqrt{2 \pi \sigma^{2}}} \exp \bigg[-\frac{(x-\mu)^{2}}{2 \sigma^{2}}\bigg]dx, \quad  k\in \mathbb{Z}.$$
 Obviously, we have the triangle inequality
\begin{eqnarray}\label{TE1}
    d_\mathrm{TV}(W,Y^{d}) \leq d_\mathrm{TV}(W,M_{i})+d_\mathrm{TV}(M_i,Y^{d}).
\end{eqnarray}
 Thus,  applying Lemma \ref{L4.1}, we have the following corollary.
\begin{corollary}\label{main2}
    For $i=1,2,3$,
    \begin{eqnarray*}
        d_\mathrm{TV}(W,Y^{d})\le d_\mathrm{TV}(W,M_{i})+\frac{C}{\sigma}.
    \end{eqnarray*}
\end{corollary}

Theorem \ref{main}, together with  \cite[Theorem 2.2]{rollin2015local}, easily provides the upper bound for the local  distance.
\begin{corollary}\label{C21} Use the same setting and notations as in Theorem $\ref{main}$.

 {\rm(i)} Assume $\mu > \sigma^{2}$, and then
   \begin{eqnarray*}
        d_{\rm loc}(W,Y^{d}) \leq  C\bigg[\frac{\gamma S(W)+p^2}{(1-2\theta_{1})q\mu}\bigg]^{1/2} [S_{2}(W)+S_{2}(Y^{d}) ]^{1/2}.
   \end{eqnarray*}

   {\rm(ii)} Assume $\mu < \sigma^{2}$, and then
   \begin{eqnarray*}
       d_{\rm loc}(W,Y^{d}) \leq  C\bigg[  (1\vee \frac{q}{p} )\frac{\gamma S(W)}{(1-2\theta_{2})\mu}\bigg]^{1/2} [S_{2}(W)+S_{2}(Y^{d}) ]^{1/2}.
   \end{eqnarray*}

  {\rm (iii)} Assume $\mu \approx \sigma^{2}$, and then
   \begin{eqnarray*}
       d_{\rm loc}(W,Y^{d}) \leq  C\bigg[\frac{\gamma S(W)}{(1-2\theta_{3})\mu}\bigg]^{1/2} [S_{2}(W)+S_{2}(Y^{d}) ]^{1/2}.
   \end{eqnarray*}
\end{corollary}
To conclude Section 1, we state a corollary for the classical $m$-dependent random variable sequence. Recall that $X_{1},\ldots, X_{n}$ is said to be
 $m$-dependent if $\{X_{i}, i \leq j \}$ is independent of $\{X_{i},i \geq j+m+1\}$ for any $j = 1,\dots,n-m-1$. Let $W=\sum_{i=1}^{n}X_{i}$, and assume that ${\rm E}X_{i} > 0$ and ${\rm E}X_{i}^{3}<\infty$. We have the following corollary.

\begin{corollary}\label{COO4} Use the same setting and notations as in Theorem $\ref{main}$.

  {\rm(i)} Assume $\mu > \sigma^{2}$, and then
$$d_\mathrm{TV}(W,M_{1})\leq C\bigg(\frac{m^{3}S(W)\sum_{i\in J}{\rm E}X_{i}^{4}+p^2}{(1-2\theta_{1})npq}\bigg).$$

{\rm(ii)} Assume $\mu < \sigma^{2}$, and then
$$d_\mathrm{TV}(W,M_{2})\leq C\bigg((p \vee q)\frac{m^{3}S(W)}{(1-2\theta_{2})rq}\sum_{i\in J}{\rm E}X_{i}^{4}\bigg).$$

{\rm(iii)} Assume $\mu \approx \sigma^{2}$, and then
$$d_\mathrm{TV}(W,M_{3})\leq C\bigg( \frac{m^{3}S(W) }{(1-2\theta_{3})(\lambda+\omega+\eta)}\sum_{i\in J}{\rm E}X_{i}^{4}\bigg).$$
\end{corollary}
\begin{remark}\label{BL1}
For most  $m$-dependent models, we usually have $S(W)/S_{2}(W) \to 1$ as $n\rightarrow \infty$.  Thus,  it is  sufficient to estimate  $S_2(W)$. Without loss of generality,  assume  $n=(3m+1)r$, where $r\in \mathbb{Z}^{+}$. Rewrite $W$ as a sum of blocks:  $$W=\sum_{i=1}^{r}Y_{i},\quad\text{where}\ Y_{i}=\sum_{j=(3m+1)(i-1)+1}^{(3m+1)i}X_{j}.$$
   Define $$\partial Y_{i}=(X_{(3m+1)(i-1)+1},X_{(3m+1)(i-1)+2},\dots,X_{(3m+1)(i-1)+m})$$  and $$Z=\{\partial Y_{1},\partial Y_{2},\ldots,\partial Y_{r}\}.$$ It is immediate that $$\mathcal{L}(W| Z=z)= \mathcal{L}(Y_{1}^{z})*\mathcal{L}(Y_{2}^{z})*\cdots*\mathcal{L}(Y_{r}^{z}),$$
   where $\mathcal{L}(Y_{i}^{z})=\mathcal{L}(Y_{i}| Z=z).$ To apply \eqref{2311} in Remark \ref{r1.3}, we need the following two additional assumptions:

  ($H_{1}$)  There exists a constant $c<2$ such that $\max_{1\leq i\leq r}S_{1}(Y_{i}^{z})\leq c$.

  ($H_{2}$)  For every $z$, $v^{\star}_{z}/V_{z}\rightarrow 0$  as $r\rightarrow \infty$.

\noindent Roughly speaking, $(H_{1})$ and $(H_{2})$ essentially require that $Y_{i}^{z}$ is not degenerate, i.e., $Y^{z}_{i}$ is not entirely determined by $\partial{Y_{i-1}}$ and $\partial{Y_{i}}$, and that all $S_{1}(Y_{i}^{z})$'s are of roughly the same order. Under the assumptions $(H_{1})$ and $(H_{2})$, both $S(W)$ and $S_{2}(W)$ are $O(N^{-1})$ by \eqref{2311}. This idea works particularly well for sums of functionals with a finite range of independent and identically distributed (i.i.d.) random variables since these sums easily satisfy  $(H_{1})$ and $(H_{2})$.
\end{remark}

To illustrate the advantages of the three-parameter approximation and explain how to estimate $S(W)$, we shall present an example of $2$-runs in detail, which is a $1$-dependent model.

\begin{example}
    Let $J =\{1,2,\ldots,N\}$, and assume that $\{I_{i}, i \in J\}$ are independent Bernoulli random variables with $I_{i} \sim B(1, \bar{p})$. To avoid edge effects, we set $I_{0}=I_{N}$. Define $$X_{i}=I_{i-1}I_{i}, \quad   W_{N}=\sum_{i=1}^{N} X_{i}.$$
 $\{X_{i},\;i\in J\}$ satisfies the $1$-dependent structure.
 Let $$A_{i}=\{j:|j-i| \leq 1\} \cap J, \quad A_{ij}=\{k:\min\{|k-i|,|k-j|\} \leq 1\} \cap J,$$  and  $$A_{ijk}=\{l:\min\{|l-i|,|l-j| ,|l-k| \leq 1\} \cap J.$$
  It is  easy to verify that the conditions (LD1)--(LD3) are satisfied for $\{X_{i}, i\in J\}$ with $A_i$, $A_{ij}$, and $A_{ijk}$. After straightforward calculations, we have $\mu=N\bar{p}^{2}$ and $\sigma^{2}=N\bar{p}^{2}+2N\bar{p}^{3}-3N\bar{p}^{4}$. Assume $\bar{p}<2/3$. Then,  we choose $M_{2}$ to approximate $W$. Further computations yield $\Gamma_{3}=6N\bar{p}^{4}-24N\bar{p}^{5}+20N\bar{p}^{6}$. Matching the first three factorial cumulants yields the parameters of $M_{2}$:
\begin{align*}
    r=\frac{(2N\bar{p}^{3}-3N\bar{p}^{4})^{3}}{(3N\bar{p}^{4}-12N\bar{p}^{5}
    +10N\bar{p}^{6})^{2}}, \quad p=\frac{2N\bar{p}^{3}-3N\bar{p}^{4}}{2N\bar{p}^{3}-12N\bar{p}^{5}
    +10N\bar{p}^{6}},\quad\lambda=N\bar{p}^{2}- \frac{(2N\bar{p}^{3}-3N\bar{p}^{4})^{2}}{3N\bar{p}^{4}-12N\bar{p}^{5}
    +10N\bar{p}^{6}},
\end{align*}
and
\begin{align}\label{TP}
   \theta_{2}=\frac{-\bar{p}+\bar{p}^{3}}{2-3\bar{p}}.
\end{align}
By \eqref{TP}, the assumption of  Corollary \ref{COO4} holds   whenever  $\bar{p}<1/2$.  Clearly,
$$\sum_{i\in J}{\rm E}X_{i}^{4}=N\bar{p}^{2}.$$

We are left to estimate $S(W)$. Without proof, we claim that $\lim_{N\rightarrow \infty}S_2(W)/S(W)=1$. As  in Remark~\ref{BL1}, we assume $n=4r$ and set $$Y_{i}=\sum_{j=4i-3}^{4i}X_{j},\quad \partial Y_{i}=X_{4i-3},\quad Z=\{X_{1},X_{5},\ldots, X_{4r-3}\}.$$
Note that each $Y_{i}$ is influenced only by $\partial Y_{i}$ and $\partial Y_{i+1}$. For each $z=\{y_{1},\dots,y_{r}\}$, through straightforward calculations, we obtain $$2(1-\bar{p})^{2}\leq S_{1}(Y_{i}^{Z=z})\leq 2-2\bar{p}^{2}.$$   It is immediate that all $v_{i}$'s  belong to $[\bar{p}^{2}, 2\bar{p}-\bar{p}^{2}]$, and by definition, $V_{z}=O(N)$ and  $v_z^{\star}= O(1)$, which give  $S_{2}(W)=O(N^{-1})$.
Recalling that $\theta_{2}<0.5$, $m=2$, $(p\vee q)\leq 1$, and $rq=O(N)$, we finally obtain, by  Corollary \ref{COO4},
that $d_\mathrm{TV}(W,M_{2})=O(N^{-1}).$
For comparison, note that Barbour and Xia \cite{barbour1999poisson}  used a two-parameter compound Poisson distribution to approximate $W$, and Brown and Xia \cite{brown2001stein} used a negative binomial distribution to approximate $W$, but their error bounds in total variation distance are only $O(N^{-0.5})$.
\end{example}

The rest of this paper is organized as follows. In Section \ref{S3}, we first briefly review the basic steps of Stein's method and provide the Stein operator associated with each $M_i$. Then, we complete the proof of Theorem 1.1 by repeatedly using Newton's expansion and controlling the error terms. In Section \ref{S22}, we consider four interesting examples, three of which are related to random graphs. Through delicate analysis, we verify the local dependence structures and determine the corresponding parameters for each example. Two claims involve elementary but lengthy computations, which are postponed to Appendices A and B. Compared with relevant results in the literature, our approximation upper bounds appear sharper.

\section{Proofs of main results}\label{S3}
\vspace{-2mm}
\subsection{Preliminaries}\label{PS1}
Suppose we are given two integer-valued random variables $U$ and $V$, where $U$ is the object of study and $V$ is the target variable.
Let $\mathcal{H}$ be a family of functions, and the distance in the sense of $\mathcal{H}$ is defined as
\begin{eqnarray}
    d_{\mathcal{H}}(U, V)=\sup_{f\in \mathcal{H}}|{\rm E}f(U)-{\rm E}f(V)|. \label{H}
\end{eqnarray}
In particular, let $\mathcal{H}=\{\mathbf{1}_{A}, A \subset \mathbb{Z}\}$ (resp. $\mathcal{H}=\{\mathbf{1}_{a}, a \in \mathbb{Z}\}$). Then,  (\ref{H}) reduces to $d_\mathrm{TV}(U, V)$ (resp. $d_{\rm loc}(U, V)$).

 Stein's method turns out to be a powerful tool in the study of $d_{\mathcal{H}}$. It usually consists of the following three steps.  The first step is to find an appropriate operator $\mathcal{A}_{V}$ that satisfies ${\rm E} \mathcal{A}_{V}g(V)= 0$ for any bounded function with $g(0)=0$ and $g(m)=0$ for $m\notin \operatorname{supp}(V)$. The second step is to find, for each $f\in \mathcal{H}$, the solution $g_f$ to the  equation
\begin{equation}\label{2.0}
    \mathcal{A}_{V} g(m)=f(m)-{\rm E} f(V), \quad m \in \mathbb{Z}
    \end{equation}
and to characterize the properties of $g_f$. The third step is to re-express the distance $d_{\mathcal{H}}(U, V)$ as
\begin{equation}\label{St}
    d_{\mathcal{H}}(U, V)=\sup _{f \in \mathcal{H}} |{\rm E} \mathcal{A}_{V} g_{f}(U) |.
\end{equation}

The following lemma provides the  Stein operator associated with $M_i$, $i=1, 2, 3$. Proofs can be found
in \cite{su2022approximation,upadhye2017Stein}. For a given discrete function, define $\Delta g(k) = g(k+1)-g(k)$.
\begin{lemma}\label{r} We have

$(1)$  for $M_{1}\sim B(n,p)*\mathcal{P}(\lambda)$,
    \begin{eqnarray}\label{L211}
        \mathcal{A}_{M_{1}}g(k)=\bigg(n\frac{p}{q}
        +\frac{\lambda}{q}-\frac{p}{q}k\bigg)g(k+1) - kg(k) +\lambda \frac{p}{q}\Delta g(k+1);
      \end{eqnarray}

$(2)$  for $M_{2}\sim NB(r,p)*\mathcal{P}(\lambda)$,
 \begin{eqnarray}\label{L212}
\mathcal{A}_{M_{2}}g(k)=q \bigg(r+\frac{\lambda p}{q}+k \bigg) g(k+1)-k g(k)-\lambda q \Delta g(k+1);
  \end{eqnarray}

$(3)$  for $M_{3}\sim\mathcal{P}(\lambda)*2\mathcal{P}(\omega /2)*3\mathcal{P}(\eta/ 3)$,
\begin{eqnarray}\label{L213}
\mathcal{A}_{M_{3}}g(k)=(\lambda+\omega+\eta) g(k+1)-k g(k)+\omega\Delta g(k+1)+\eta[\Delta g(k+1)+\Delta g(k+2)].
    \end{eqnarray}
\end{lemma}

Let $g_{A}$ (resp. $g_{a}$)  denote  the solution to (\ref{2.0}) corresponding to $f=\mathbf{1}_{A}$  (resp. $\mathbf{1}_{a}$), where $A\subset \mathbb{Z}$ (resp. $a \in \mathbb{Z})$. Let $\|\cdot\|$ represent the maximum parametric value. The following lemma provides consistent upper bounds for $\|\Delta g_{A}\|$ and $\|g_{a}\|$, which can be found in \cite[Lemma 2.3]{su2022approximation} and \cite[Lemma 3.1]{upadhye2017Stein}, respectively.

\begin{lemma}\label{s} We have

{\rm(i)} for $M_{1}\sim B(n,p)*\mathcal{P}(\lambda)$, if $2 \lambda  < q\lfloor n+ \lambda/p\rfloor$,
\begin{eqnarray*}
   \|\Delta g_{A}\|\leq \frac{q}{\lfloor n+\lambda/p\rfloor pq-2 \lambda p},\quad \| g_{a}\|\leq  \frac{q}{\lfloor n+\lambda/p\rfloor pq-2 \lambda p};
\end{eqnarray*}

{\rm(ii)} for $M_{2}\sim NB(r,p)*\mathcal{P}(\lambda)$, if $2 \lambda  q <(rq+\lambda p)$,
\begin{eqnarray*}
    \|\Delta g_{A}\|\leq \frac{1}{(rq+\lambda p)-2 \lambda q},\quad \| g_{a}\|\leq  \frac{1}{(rq+\lambda p)-2 \lambda q};
\end{eqnarray*}

{\rm(iii)} for $M_{3}\sim\mathcal{P}(\lambda)*2\mathcal{P}(\omega /2)*3\mathcal{P}(\eta /3)$, if $2(\omega+2\eta)< \lambda+\omega+\eta$,
\begin{eqnarray*}
    \|\Delta g_{A}\|\leq \frac{1}{\lambda+\omega+\eta-2 (\omega+2\eta) },\quad \| g_{a}\|\leq \frac{1}{\lambda+\omega+\eta-2(\omega+2\eta)}.
\end{eqnarray*}
\end{lemma}

We finally need the following lemma concerning the normal approximation of $M_{i}$, which is a direct consequence of \cite[Theorem 7.3]{chen2011normal} together with the unimodality of $M_{i}$.
\begin{lemma}\label{L4.1} For $M_{i}$, $i=1,2,3$, we have
        $$d_\mathrm{TV} (M_{i}, Y^{d} )\leq \frac{C}{\sqrt{\operatorname{Var}W}}.$$
\end{lemma}

\subsection{Proof of Theorem \ref{main}} \label{SS32}
Let
$$W_{i}=W-\sum_{l\in A_{i}}X_{l},\quad W_{ij}=W-\sum_{l\in A_{ij}}X_{l},\quad W_{ijk}=W-\sum_{l\in A_{ijk}}X_{l},$$
where $A_{i}, A_{ij}$, and $A_{ijk}$ are defined as in (LD1)--(LD3).

Let us begin with the proof of (iii). By Lemma \ref{r} and the independence of $X_{i}$ and $W_{i}$, we have
\begin{align}\label{1123}
   {\rm E}\mathcal{A}_{M_{3}}g(W)&=(\lambda+\omega+\eta){\rm E} g(W+1)-{\rm E}W g(W)+\omega {\rm E}\Delta g(W+1)+\eta{\rm E}[\Delta g(W+1)+\Delta g(W+2)]\notag\\
    &=\sum_{i\in J} {\rm E}X_{i} {\rm E}g(W+1)-\sum_{i\in J}{\rm E}X_{i}g(W)+\omega{\rm E}\Delta g(W+1)+\eta{\rm E}[\Delta g(W+1)+\Delta g(W+2)]\notag\\
    &=\sum_{i\in J} {\rm E}X_{i} {\rm E}[g(W+1)-g(W_{i}+1)]-\sum_{i\in J}{\rm E}X_{i} [g(W)-g(W_{i}+1)]\notag\\
    &\quad\,+ \{\omega{\rm E}\Delta g(W+1)+\eta{\rm E}[\Delta g(W+1)+\Delta g(W+2)] \}.
\end{align}

Now, we analyze each term on the right-hand side of (\ref{1123}) separately. Start with $\sum_{i\in J} {\rm E}X_{i} {\rm E}[g(W+1)-g(W_{i}+1)]$. Using Newton's expansion
 \cite[p.\,518]{barbour2002total}, we have
 $$ g(w+s)=g(w+1)+(s-1) \Delta g(w+1)+\frac{(s-1)(s-2)}{2}\Delta^{2} g(w+1)+\sum_{u=1}^{s-1}\sum_{v=1}^{k-2}(u-v-1) \Delta^{3} g(w+v),$$
and thus
\begin{align}\label{t1t1t1}
    &{\rm E}[g(W+1)-g(W_{i}+1)]\notag\\
    &\quad={\rm E}\sum_{j\in A_{i}}X_{j}\Delta g(W_{i}+1)+\frac{1}{2}{\rm E}\sum_{j\in A_{i}}X_{j}\bigg(\sum_{k\in A_{i}}X_{k}-1\bigg)\Delta^{2} g(W_{i}+1)+E_{1},
\end{align}
where the error term $E_1$ is defined as
\begin{align}\label{E1}
    E_{1}= {\rm E}\sum_{u=1}^{\Sigma_{j \in A_{i}}X_{j}}\bigg[\sum_{v=1}^{u-2}(u-v) \Delta^{3} g(W_{i}+v)\bigg].
\end{align}
Let $p^{*}_{l}={\rm P}(W_{ijk}=l| X_{A_{ijk}})$. Then,  for every fixed $v$,
    \begin{align*}
        &{\rm E} (\Delta^{3} g(W_{i}+v)| X_{A_{ijk}} )\\
        &\quad=\sum_{l=0}^{\infty}p^{*}_{l} [\Delta g(X_{A_{ijk}}^{*}+l+v+2)-2\Delta g(X_{A_{ijk}}^{*}+l+v+1)+\Delta g(X_{A_{ijk}}^{*}+l+v) ]\\
        &\quad= \sum_{l=0}^{\infty} \Delta g(X_{A_{ijk}}^{*}+l+v)(p^{*}_{l}-2p^{*}_{l-1}+p^{*}_{l-2}).
    \end{align*}
  Hence, using the fact that   $S_{2}(W| X_{A_{ijk}})=\sum_{l=0}^{\infty} |p^{*}_{l}-2p^{*}_{l-1}+p^{*}_{l-2}|$, we get
    \begin{align*}
        |  {\rm E} (\Delta^{3} g(W_{i}+v)| X_{A_{ijk}} ) | &\leq \|\Delta g\|  \sum_{l=0}^{\infty} |p^{*}_{l}-2p^{*}_{l-1}+p^{*}_{l-2}|\\
     &=   \|\Delta g\| S_{2}(W| X_{A_{ijk}})\\
     &\leq   \|\Delta g\| S(W),
    \end{align*}
    where the last inequality follows from the fact that $S_{2}(W| X_{A_{ijk}})\le S(W)$.
    Summing over $u$ and $v$, we get
 \begin{align}
   E_{1}&=  {\rm E}\sum_{u=1}^{\Sigma_{j \in A_{i}}X_{j}}\bigg[\sum_{v=1}^{u-2}(u-v) {\rm E} (\Delta^{3} g(W_{i}+v)| X_{A_{ijk}} )\bigg]\nonumber\\
   &\leq  \|\Delta g\| S(W){\rm E}\bigg(X_{i}\sum_{j \in A_{i}}X_{j}\sum_{k \in A_{i}}X_{k}\sum_{l \in A_{i}}X_{l}\bigg).\nonumber
\end{align}
For the first term on the right-hand side  of (\ref{t1t1t1}), note that for $j\in A_{i}$,
\begin{align}\label{2.10}
     \Delta g(W_{i}+1)&= {\rm E}\Delta g(W+1)+[\Delta g(W_{i}+1) - {\rm E}\Delta g(W_{ij}+1)]\notag\\
     & \quad\,+[{\rm E}\Delta g(W_{ij}+1)- {\rm E}\Delta g(W+1)].
\end{align}
Thus, by the independence of $\{X_{i}, X_{j}\}$ and $W_{ij}$, we have
\begin{align}\label{ttt1}
{\rm E}\sum_{j\in A_{i}}X_{j}\Delta g(W_{i}+1)&= {\rm E}\sum_{j\in A_{i}}X_{j}{\rm E}\Delta g(W+1)+{\rm E}\sum_{j\in A_{i}}X_{j}[\Delta g(W_{i}+1)-\Delta g(W_{ij}+1)]\notag\\
&\quad\, +{\rm E}\sum_{j\in A_{i}}X_{j}{\rm E}[\Delta g(W_{ij}+1)-\Delta g(W+1)]\notag\\
&= {\rm E}\sum_{j\in A_{i}}X_{j}{\rm E}\Delta g(W+1)+{\rm E}\sum_{j\in A_{i}}X_{j}\sum_{k\in A_{ij}\backslash A_{i}}X_{k}\Delta^{2} g(W_{ij}+1)\notag\\
&\quad\, -{\rm E}\sum_{j\in A_{i}}X_{j}{\rm E}\sum_{k\in A_{ij}}X_{k}\Delta^{2}g(W_{ij}+1)+E_{2},
\end{align}
where $E_2$ denotes the error term. By a similar argument, we have
\begin{align*}
    E_{2}\leq  \|\Delta g\| S(W){\rm E}\bigg(\sum_{j\in A_{i}}X_{j}{\rm E}\sum_{k\in A_{ij}}X_{k}\sum_{l\in A_{ij}}X_{l}+\sum_{k\in A_{ij}\backslash A_{i}}X_{k}\sum_{l\in A_{ij}\backslash A_{i}}X_{l}\bigg).
\end{align*}
Also, note the following equality:
\begin{align*}
     \Delta^{2} g(W_{ij}+1)&= {\rm E}\Delta^{2} g(W+1)+[\Delta^{2} g(W_{ij}+1) - {\rm E}\Delta^{2} g(W_{ijk}+1)]\notag\\
     &\quad\, +[{\rm E}\Delta^{2} g(W_{ijk}+1)- {\rm E}\Delta^{2} g(W+1)].
\end{align*}
Thus, by the independence of $\{X_{i}, X_{j},X_{k}\}$ and $W_{ijk}$,
\begin{align}\label{ttt21}
    {\rm E}\bigg[\sum_{j\in A_{i}}X_{j}\sum_{k\in A_{ij}\backslash A_{i}}X_{k}\Delta^{2} g(W_{ij}+1)\bigg]= {\rm E}\bigg[\sum_{j\in A_{i}}X_{j}\sum_{k\in A_{ij}\backslash A_{i}}X_{k}\bigg]{\rm E}\Delta^{2} g(W+1)+E_{3}
 \end{align}
 and
 \begin{align}\label{ttt22}
    {\rm E}\sum_{j\in A_{i}}X_{j}{\rm E}\sum_{k\in A_{ij}}X_{k}\Delta^{2}g(W_{ij}+1)&= {\rm E}\sum_{j\in A_{i}}X_{j}{\rm E}\sum_{k\in A_{ij}}X_{k}{\rm E}\Delta^{2}g(W+1)+E_{4},
   \end{align}
where $E_3$ and $E_4$ denote the error terms.
Using similar arguments to derive the upper bound for $E_1$, we have
 \begin{align*}
    E_{3}\leq  \|\Delta g\| S(W)\bigg[{\rm E}\bigg(\sum_{j\in A_{i}}X_{j}\sum_{k\in A_{ij}}X_{k}\bigg){\rm E}\sum_{l\in A_{ijk}}X_{l}+{\rm E}\bigg(\sum_{j\in A_{i}}X_{j}\sum_{k\in A_{ij}}X_{k}\sum_{l\in A_{ijk}\backslash A_{ij}}X_{l}\bigg)\bigg]
\end{align*}
and
\begin{align*}
    E_{4}\leq  \|\Delta g\| S(W)\bigg[{\rm E} \sum_{j\in A_{i}}X_{j}{\rm E}\sum_{k\in A_{ij}}X_{k}{\rm E}\sum_{l\in A_{ijk}}X_{l} +{\rm E}\sum_{j\in A_{i}}X_{j}{\rm E}\bigg(\sum_{k\in A_{ij}}X_{k}\sum_{l\in A_{ijk}\backslash A_{ij}}X_{l}\bigg)\bigg].
\end{align*}
Substituting  (\ref{ttt21}) and (\ref{ttt22}) into (\ref{ttt1}) yields
\begin{align}\label{ttt2}
{\rm E}\sum_{j\in A_{i}}X_{j}\Delta g(W_{i}+1)&= {\rm E}\sum_{j\in A_{i}}X_{j}{\rm E}\Delta g(W+1)+{\rm E}\bigg[\sum_{j\in A_{i}}X_{j}\sum_{k\in A_{ij}\backslash A_{i}}X_{k}\bigg]{\rm E}\Delta^{2} g(W+1)\notag\\
&\quad\, -{\rm E}X_{i}{\rm E}\sum_{j\in A_{i}}X_{j}{\rm E}\sum_{k\in A_{ij}}X_{k}{\rm E}\Delta^{2}g(W+1)+E_{2}+E_{3}+E_{4}.
\end{align}
For the second term on the right-hand side  of (\ref{t1t1t1}), a similar argument gives
\begin{align}\label{txyi} {\rm E}\bigg[\sum_{j\in A_{i}}X_{j}\bigg(\sum_{k\in A_{i}}X_{k}-1\bigg)\Delta^{2} g(W_{i}+1)\bigg]={\rm E}\bigg[\sum_{j\in A_{i}}X_{j}\bigg(\sum_{k\in A_{i}}X_{k}-1\bigg){\rm E}\Delta^{2} g(W+1)\bigg]+E_{5},
\end{align}
where the error term $E_5$ is controlled by
\begin{align*}
    E_{5}\leq \|\Delta g\| S(W)\bigg[{\rm E}\sum_{j\in A_{i}}X_{j}\bigg(\sum_{k\in A_{i}}X_{k}-1\bigg){\rm E}\sum_{k\in A_{ij}}X_{k}+{\rm E}\sum_{j\in A_{i}}X_{j}\bigg(\sum_{k\in A_{i}}X_{k}-1\bigg)\sum_{k\in A_{ij}\backslash A_{i}}X_{k}\bigg].
\end{align*}
Putting (\ref{ttt1})--(\ref{txyi}) together and noting that $E_{1}, \ldots, E_{5}$ are each bounded by $\gamma_{i}\|\Delta g\|S(W)$, we have
\begin{align}\label{T1T1}
      &{\rm E}X_{i} {\rm E}[g(W+1)-g(W_{i}+1)]\notag\\
     &\quad\leq {\rm E} X_{i}{\rm E}\sum_{j\in A_{i}}X_{j}{\rm E}\Delta g(W+1)+{\rm E} X_{i}{\rm E}\bigg[\sum_{j\in A_{i}}X_{j}\sum_{k\in A_{ij}\backslash A_{i}}X_{k}\bigg]{\rm E}\Delta^{2} g(W+1)\notag\\
&\quad\quad -{\rm E} X_{i}{\rm E}\sum_{j\in A_{i}}X_{j}{\rm E}\sum_{k\in A_{ij}}X_{k}{\rm E}\Delta^{2}g(W+1) +\frac{1}{2}{\rm E}X_{i}{\rm E}\bigg[\sum_{j\in A_{i}}X_{j}\bigg(\sum_{k\in A_{i}}X_{k}-1\bigg)\bigg]{\rm E}\Delta^{2} g(W+1)\notag\\
&\quad\quad +O (\gamma_{i}\|\Delta g\| S(W)) ).
\end{align}

 Next, we turn to the second term on the right-hand side  of (\ref{1123}), i.e., $ \sum_{j\in J}{\rm E}X_{j} [g(W)-g(W_{j}+1)]$. Using Newton's expansion again and  similar calculations yields
\begin{align}\label{T1T2}
    &{\rm E}X_{i}  [g(W)-g(W_{i}+1) ]\notag\\
    &\quad\leq  {\rm E}\bigg[X_{i}\bigg(\sum_{j\in A_{i}}X_{j}-1\bigg)\bigg]{\rm E}\Delta g(W+1)     +{\rm E}\bigg[X_{i}\bigg(\sum_{j\in A_{i}}X_{j}-1\bigg)\sum_{k\in A_{ij}\backslash A_{i}}X_{k}\bigg]{\rm E}\Delta^{2} g(W+1)\notag\\
    &\quad\quad -{\rm E}\bigg[X_{i}\bigg(\sum_{j\in A_{i}}X_{j}-1\bigg)\bigg]{\rm E}\sum_{k\in A_{ij}}X_{k}{\rm E}\Delta^{2} g(W+1)\notag\\
    &\quad\quad +\frac{1}{2}{\rm E}\bigg[X_{i}\bigg(\sum_{j\in A_{i}}X_{j}-1\bigg)\bigg(\sum_{j\in A_{i}}X_{j}-2\bigg)\bigg]{\rm E}\Delta^{2} g(W+1) +O(\gamma_{i}\|\Delta g\|S(W)).
\end{align}
In addition, noting that $\Delta g(W+2)=\Delta g(W+1)+\Delta^{2} g(W+1)$,     the third term on the right-hand side  of (\ref{1123}) can be written as
\begin{align}\label{T1T3}
    \omega{\rm E}\Delta g(W+1)+ \eta{\rm E}[\Delta g(W+1)+\Delta g(W+2)]= (\omega+2\eta){\rm E}\Delta g(W+1)+\eta {\rm E}\Delta^{2} g(W+1).
\end{align}
Define
\begin{align}\label{G11}
G_{1}&= \sum_{i\in J}\bigg[{\rm E}X_{i}{\rm E}\sum_{j\in A_{i}}X_{j}-{\rm E}X_{i}\bigg(\sum_{j\in A_{i}}X_{j}-1\bigg)\bigg],\\
    G_{2}&= {\rm E}X_{i}{\rm E}\bigg[\sum_{j\in A_{i}}X_{j}\sum_{k\in A_{ij}\backslash A_{i}}X_{k}\bigg]-{\rm E}\bigg[X_{i}\bigg(\sum_{j\in A_{i}}X_{j}-1\bigg)\sum_{k\in A_{ij}\backslash A_{i}}X_{k}\bigg]\notag\\\label{G22}
&\quad\, -{\rm E}X_{i}{\rm E}\sum_{j\in A_{i}}X_{j}{\rm E}\sum_{k\in A_{ij}}X_{k}+{\rm E}\bigg[X_{i}\bigg(\sum_{j\in A_{i}}X_{j}-1\bigg)\bigg]{\rm E}\sum_{k\in A_{ij}}X_{k}\notag\\
&\quad\, +\frac{1}{2}{\rm E}X_{i}{\rm E}\bigg[\sum_{j\in A_{i}}X_{j}\bigg(\sum_{k\in A_{i}}X_{k}-1\bigg)\bigg]-\frac{1}{2}{\rm E}\bigg[X_{i}\bigg(\sum_{j\in A_{i}}X_{j}-1\bigg)\bigg(\sum_{j\in A_{i}}X_{j}-2\bigg)\bigg].
\end{align}
Combining (\ref{T1T1})--(\ref{T1T3}), we  get
\begin{align}\label{BLE1}
{\rm E}\mathcal{A}_{M_{3}}g(W)\leq (G_{1}+\omega+2\eta){\rm E}\Delta g(W+1)+(G_{2}+\eta){\rm E}\Delta^{2} g(W+1)+O (\gamma \|\Delta g\| S(W) ).
\end{align}

To complete the proof, we need a lemma, which is one of the most important findings of this paper.
\begin{lemma}\label{LEMA4}It holds that
\begin{equation}\label{LE1}
    G_{1}=-\Gamma_{2},\quad  G_{2}=-\frac{\Gamma_{3}}{2}.
\end{equation}
\end{lemma}
\begin{proof}
Note that for every $B\supseteq A_{i}$,
\begin{equation}\label{G121}
    \sum_{i\in J}\sum_{j\in B}{\rm E}X_{i}X_{j}-\sum_{i\in J}\sum_{j\in B}{\rm E}X_{i}{\rm E}X_{j}=\operatorname{Var}W .
\end{equation}
Taking $A_{i}$  as  $B$, we get
\begin{align*}
    G_{1}&= \sum_{i\in J}\bigg[{\rm E}X_{i}{\rm E}\sum_{j\in A_{i}}X_{j}-{\rm E}X_{i}\bigg(\sum_{j\in A_{i}}X_{j}-1\bigg)\bigg]\\
    &= -\operatorname{Var}W+{\rm E}W=-\Gamma_{2}.
\end{align*}
 It remains to prove $G_{2}=-\frac{\Gamma_{3}}{2}$. Rearrange $G_{2}$ as follows:
\begin{align}\label{11}
 G_{2}&= -\sum_{i\in J}{\rm E}X_{i}
 +\sum_{i\in J}\sum_{k\in A_{ij}}({\rm E}X_{i}X_{k}-{\rm E}X_{i}{\rm E}X_{k})+\sum_{i\in J}\sum_{j\in A_{i}}\bigg (\frac{1}{2}{\rm E}X_{i}X_{j}-\frac{1}{2}{\rm E}X_{i}{\rm E}X_{j} \bigg )\notag\\
&\quad\, +\sum_{i\in J}\sum_{j\in A_{i}}\sum_{k\in A_{ij}}({\rm E}X_{i}{\rm E}X_{j}X_{k}-{\rm E}X_{i}{\rm E}X_{j}{\rm E}X_{k}-{\rm E}X_{i}X_{j}X_{k}+{\rm E}X_{i}X_{j}{\rm E}X_{k})\notag\\
 &\quad\, +\sum_{i\in J}\sum_{j,k \in A_{i}} \bigg(-\frac{1}{2}{\rm E}X_{i}{\rm E}X_{j}X_{k}+\frac{1}{2}{\rm E}X_{i}X_{j}X_{k} \bigg).
\end{align}
 Taking $B$ as $A_{ij}$ in (\ref{G121}), we have
\begin{align}\label{1132}
    \sum_{i\in J}\sum_{k\in A_{ij}}({\rm E}X_{i}X_{k}-{\rm E}X_{i}{\rm E}X_{k})+\sum_{i\in J}\sum_{j\in A_{i}} \bigg(\frac{1}{2}{\rm E}X_{i}X_{j}-\frac{1}{2}{\rm E}X_{i}{\rm E}X_{j} \bigg )=\frac{3}{2}\operatorname{Var}W.
\end{align}
By the independence of $X_{i}$ and $\{X_{j}, J\notin A_{i}\}$,
\begin{align}\label{11123}
    & \sum_{i\in J}\sum_{j\in A_{i}}\sum_{k\in A_{ij}} ({\rm E}X_{i}{\rm E}X_{j}X_{k}-{\rm E}X_{i}{\rm E}X_{j}{\rm E}X_{k}-{\rm E}X_{i}X_{j}X_{k}+{\rm E}X_{i}X_{j}{\rm E}X_{k} )\notag\\
    &\quad= \sum_{i\in J}\sum_{j\in A_{i}} ({\rm E}X_{i}{\rm E}X_{j}W-{\rm E}X_{i}{\rm E}X_{j}{\rm E}W-{\rm E}X_{i}X_{j}W+{\rm E}X_{i}X_{j}{\rm E}W )\notag\\
    &\quad= \sum_{i\in J}\sum_{j\in A_{i}} ({\rm E}X_{i}{\rm E}X_{j}W-{\rm E}X_{i}X_{j}W )+(\operatorname{Var}W){\rm E}W.
\end{align}
Substituting  (\ref{1132}) and (\ref{11123}) into (\ref{11}), we see that
\begin{align}\label{T1T4}
    G_{2}&= \sum_{i\in J}\sum_{j\in A_{i}} ({\rm E}X_{i}{\rm E}X_{j}W-{\rm E}X_{i}X_{j}W )+\sum_{i\in J}\sum_{j,k\in A_{i}} \bigg(-\frac{1}{2}{\rm E}X_{i}{\rm E}X_{j}X_{k}+\frac{1}{2}{\rm E}X_{i}X_{j}X_{k} \bigg)\notag\\
    &\quad\, +(\operatorname{Var}W){\rm E}W+\frac{3}{2}\operatorname{Var}W-{\rm E}W.
\end{align}
In addition, noting that $W=X_{A_{i}}^{*}+X_{A_{i}^{c}}^{*}$, we have
\begin{align}\label{10}
& \sum_{i\in J}\sum_{j\in A_{i}} ({\rm E}X_{i}{\rm E}X_{j}W-{\rm E}X_{i}X_{j}W )+\sum_{i\in J}\sum_{j,k\in A_{i}} \bigg(-\frac{1}{2}{\rm E}X_{i}{\rm E}X_{j}X_{k}+\frac{1}{2}{\rm E}X_{i}X_{j}X_{k}\bigg )\notag\\
&\quad= \frac{1}{2}\sum_{i\in J} [{\rm E}X_{i}(X_{A_{i}}^{*})^{2}-{\rm E}X_{i}{\rm E}(X_{A_{i}}^{*})^{2}+2({\rm E}X_{i}{\rm E}X_{A_{i}}^{*}W - {\rm E}X_{i}X_{A_{i}}^{*}W) ]\notag\\
   &\quad=  \frac{1}{2}\sum_{i\in J} [-{\rm E}X_{i}(X_{A_{i}}^{*})^{2}+{\rm E}X_{i}{\rm E}(X_{A_{i}}^{*})^{2}- 2{\rm E}X_{i}X_{A_{i}}^{*}X_{A_{i}^{c}}^{*}+2{\rm E}X_{i}{\rm E}X_{A_{i}}^{*}X_{A_{i}^{c}}^{*} ]\notag\\
   &\quad=   \frac{1}{2}\sum_{i\in J} [- {\rm E}X_{i}(X_{A_{i}}^{*}+X_{A_{i}^{c}}^{*})^{2}+{\rm E}X_{i}{\rm E}(X_{A_{i}}^{*}+X_{A_{i}^{c}}^{*})^{2} ] \notag\\
   &\quad=  \frac{1}{2}\sum_{i\in J} ( -{\rm E}X_{i}W^{2}+{\rm E}X_{i}{\rm E}W^{2} )\notag\\
   &\quad=-\frac{{\rm E}W^{3}}{2}+\frac{{\rm E}W{\rm E}W^{2}}{2}.
\end{align}
Substituting (\ref{10}) into (\ref{T1T4}), we get
\begin{align*}
    G_{2}&= -\frac{{\rm E}W^{3}}{2}+\frac{{\rm E}W{\rm E}W^{2}}{2}+(\operatorname{Var}W){\rm E}W+\frac{3}{2}\operatorname{Var}W-{\rm E}W\\
    &= -\frac{\Gamma_{3}}{2}-\frac{3}{2}{\rm E}W({\rm E}W^{2}+{\rm E}W+\Gamma_{2})-\frac{3}{2} ({\rm E}W^{2}+{\rm E}W+\Gamma_{2})+{\rm E}W^{3}+\frac{3}{2}{\rm E}W^{2}+{\rm E}W\\
    &\quad\, +\frac{{\rm E}W({\rm E}W^{2}+{\rm E}W+\Gamma_{2})}{2}+({\rm E}W+\Gamma_{2}) {\rm E}W+\frac{3}{2}({\rm E}W+\Gamma_{2})-{\rm E}W\\
    &= -\frac{\Gamma_{3}}{2},
\end{align*}
as desired.
\end{proof}

  We now continue to prove Theorem \ref{main}. Substituting (\ref{LE1}) into (\ref{BLE1}) and using the relationship~(\ref{k3}),  we  obtain
$${\rm E}\mathcal{A}_{M_{3}}g(W)=O (\gamma \|\Delta g\| S(W) ).$$
Replacing $g$ with $g_A,$ where $ A\subset \mathbb{Z}^{+}$, and  using
 Lemma~\ref{s}, we complete the proof of (iii).

Now, turn to (i) and (ii).  Following the line of the proof of (iii), one can prove
$${\rm E}\mathcal{A}_{M_{1}}g(W)=\bigg(G_{1}-np^{2}-\frac{\delta p^2}{q}\bigg){\rm E}\Delta g(W+1)+ \bigg(\frac{G_{2}}{q}+\frac{np^{3}}{q} \bigg){\rm E}\Delta^{2}g(W+1)+
 \frac{1}{q} O(\gamma \|\Delta g\|S(W))
$$
and
$${\rm E}\mathcal{A}_{M_{2}}g(W)=\bigg(G_{1}+\frac{rq^{2}}{p^{2}}\bigg){\rm E}\Delta g(W+1)+
 \bigg(pG_{2}+\frac{rq^{3}}{p^{2}} \bigg){\rm E}\Delta^{2}g(W+1)+
 (p\vee q) O(\gamma \|\Delta g\|S(W)),$$
where $\delta$ is defined as in (\ref{K1}), and $G_{1}$ and $G_{2}$ are defined as in (\ref{G11}) and (\ref{G22}).
 Applying Lemma \ref{LEMA4}, we obtain
$${\rm E}\mathcal{A}_{M_{1}}g(W)=
 \frac{1}{q} O(\gamma \|\Delta g\| S(W)+p^2\delta {\rm E}\Delta g(W+1))
$$
and
$${\rm E}\mathcal{A}_{M_{2}}g(W)=
(p\vee q)O(\gamma \|\Delta g\|S(W)),$$
which, together with Lemma \ref{s},   concludes the proofs of (i) and (ii).

\section{Applications}\label{S22}
\vspace{-2mm}
\subsection{Counting vertices where all edges point inward}
In this subsection, we consider a problem studied
by  Arratia et al.~\cite{arratia1989two}. Let ${\cal V}=\{0,1\}^{d}$ and ${\cal E}=\{(u,v): u,v \in {\cal V},\,\text{and}\ d_{H}(u,v)=1\}$, where $d\ge 2$ and $d_{H}$ denotes the Hamming distance in ${\cal V}$. Assume that each edge in ${\cal E}$ is assigned a random
direction by tossing a fair coin. Denote by $J$  the set of all $2^{d}$ vertices, and let  $X_{i}$ be the indicator that all edges of vertex $i$  point inward.

 \begin{figure}[h!]
\centering
\begin{tikzpicture}
  \path [draw=black,postaction={on each segment={mid arrow=black}}]
(0,0)-- (2,1)
(1,2)--(2,1)
(-1,1)--(1,2)
 (-1,1)-- (0,0)
(0,-2)--(0,0)
(2,-1)--(0,-2)
 (1,0)--(2,-1)
 (-1,-1)--(1,0)
(-1,-1)-- (0,-2)
(2,-1)--(2,1)
(1,0)--(1,2)
 (-1,1)-- (-1,-1);
\end{tikzpicture}
  \caption{An configuration in $\{0,1\}^{3}$}
    \label{FI13}
\end{figure}

Set
$$A_{i}=\{j:|j-i|=1\},\quad A_{ij}=\{k:j\in A_{i};\,|k-i|=1\;{\rm or} \;|k-j|=1\},$$
and
$$A_{ijk}=\{l: j\in A_{i};\, k\in A_{ij}; \,|l-i|=1,|l-j|=1\; {\rm or} \;|l-k|=1\}.$$
Obviously, $\sharp A_{i}=d$, $\sharp A_{ij}\le 2d$, and $\sharp A_{ijk}\le 3d$. It is also easy to verify that the conditions (LD1)--(LD3) are satisfied for $\{X_{i}, i\in J\}$ with $A_i, A_{ij}$, and $A_{ijk}$.  We are concerned with the number of vertices at which
all $d$ edges point inward. In particular, set $W =\sum_{i\in J}X_{i}$.  Figure 1 illustrates a configuration in $\{0, 1\}^3$ with $W=1$.
\begin{theorem}\label{T2}
Under the setting described above, we have
\begin{align}\label{R1}
   d_\mathrm{TV}(W, M_3) = O(d^{3}2^{-3d}),
\end{align}
where $\{\lambda,\omega,\eta\}$ satisfy the following equations$:$
\begin{equation}\label{616}
     \lambda+\omega+\eta=1,\quad\omega+2\eta=(d-1)2^{-d},
     \quad2\eta=(3d^{2}+3d+2)2^{-2d+1}.
 \end{equation}
\end{theorem}
\begin{remark}
     Arratia et al.~\cite{arratia1989two}  used ${\cal P}(1)$ to approximate $W$ and obtained
    \begin{equation}\label{R2}
        d_\mathrm{TV}(W,\mathcal{P}(1))\leq d2^{-d}.
    \end{equation}
Observably,  (\ref{R1}) is better than (\ref{R2}) in the numerical approximation of probability  at the cost of calculating the  probability value of the three-parameter model ${\cal P}(\lambda)  \ast 2{\cal P}(\omega/2)\ast 3{\cal P}(\eta/3)$.
\end{remark}
\begin{proof}[Proof of Theorem $\ref{T2}$]
By simple but lengthy calculations,
 $$ \Gamma_{1}(W)=1, \quad \Gamma_{2}(W)=(d-1)2^{-d},\quad \Gamma_{3}(W)=(3d^{2}+3d+2)2^{-2d+1}.$$
 We use $M_{3}={\cal P}(\lambda)  \ast 2{\cal P}(\omega/2)\ast 3{\cal P}(\eta/3)$ to approximate $W$, where the three parameters $\lambda$, $\omega$, and $\eta$ are exactly determined by (\ref{616}).

To apply  Theorem \ref{main}(iii), we need to further control $\gamma$, $S(W)$, and $\theta_3$. First, it is immediate that $S(W)\le 4$.
Also, it follows that $\theta_3=(d-1)2^{-d}$. Finally, notice that ${\rm E}X_i= 2^{-d}$ and ${\rm E}X_iX_j=0$ whenever $d_{H}(i,j)=1$. Thus, it follows that $\gamma \leq Cd^{3}2^{-3d}$.

We now complete the proof of Theorem \ref{T2}.
\end{proof}

\subsection{Birthday problem}
In this subsection, we consider the well-known birthday problem.  It can be equivalently described as follows: Fix $k\ge 2$, randomly place  $n$ balls into $d$ boxes  with equal probability, and estimate the probability that at least one box receives $k$ or more balls.

Let $J=\{i:i \in \{1,2,\ldots,n\},  \, \sharp i =k\}$. For each $i\in J$, denote
$$X_{i}=\left\{\!
\begin{aligned}
1 & \quad \text{if the balls indexed by $i$ all go into the same box,}& \\
0  &\quad \text{otherwise}.& \\
\end{aligned}
\right.$$
Set
\begin{align*}
    & A_{i}=\{j:j\cap i \neq \emptyset\},\quad A_{ij}=\{m:m\cap i \neq \emptyset\ \text{or}\ m\cap j \neq \emptyset\},\\
   &    A_{ijm}=\{l:l\cap i \neq \emptyset, l\cap j \neq \emptyset,\ \text{or}\ l\cap m \neq \emptyset\}.
\end{align*}
Clearly, $\sharp(A_i)$, $\sharp(A_{ij})$, and $\sharp(A_{ijm})$ are all $O(n^{k-1})$. It is immediate that $\{X_i, i\in J\}$  satisfies  (LD1)--(LD3) with $A_{i}$, $A_{ij}$, and $A_{ijm}$, and that $p:={\rm P}(X_{i}=1)=d^{1-k}$. Define $W=\sum_{i\in J}X_{i}$. Thus, $${\rm P}(\text{no box gets}\; k \; \text{or more balls}) = {\rm P}(W = 0).$$
Simple calculations give
$$  m_1={\rm E}W= \sum_{i\in J}{\rm E}X_{i}=\binom n kd^{1-k} $$
  and
  \begin{align*}
    m_{2}&= {\rm E}W^2={\rm E}\sum_{i\in J}X_{i}\bigg[X_{i}+\sum_{j\in A_{i}\backslash\{i\}}X_{j}+\sum_{J\backslash A_{i}}X_{j}\bigg]\\
    &= m_{1}+\sum_{i\in J}\sum_{j=1}^{k-1}\binom k j\binom {n-k}{k-j} d^{1+j-2 k}+\binom nk\binom {n-k}kd^{2-2k}.
\end{align*}
By $\Gamma_{2}(W)=m_{2}-m_{1}^{2}-m_{1}$, we have
\begin{align*}
    \Gamma_{2} (W)=\sum_{i\in J}\sum_{j=1}^{k-1}\binom k j\binom {n-k}{k-j} d^{1+j-2 k}+\binom n k\bigg[\binom {n-k}k-\binom n k\bigg] d^{2-2k}.
\end{align*}
\begin{lemma}\label{L3.5}
    Let $$a_j=\sum_{l=1}^{j}\sum_{m=(j-l)\vee 1}^{k-l}\binom{j}{m}\binom{j}{l}\binom{k-j}{k-l-m}.$$ Define $$\Tilde{\Gamma}_{1}=\frac{n^{k}}{k!d^{k-1}},\quad\Tilde{\Gamma}_{2}=\sum_{j=2}^{k-1}\binom{k}{j}\frac{n^{2k-j}}{k!(k-j)!d^{2k-j-1}},\quad\Tilde{\Gamma}_{3}=3\sum_{j=2}^{k-1}\binom{k}{j}\frac{a_jn^{2k-j}}{k!(k-j)!d^{2k-j-1}}.$$
   Suppose both $n$ and $d$ tend  to $\infty$ such that $n^{k} \asymp d^{k-1}$. We have    for $i=1,2,3$,
    \begin{equation}
      | \Gamma_{i}(W)-\Tilde{\Gamma}_{i}|\leq \frac{C}{n}. \notag
    \end{equation}
\end{lemma}
\begin{proof}
    It is easy to calculate
 \begin{equation*}\label{G1}
        |\Gamma_{1}(W)-\Tilde{\Gamma}_{1}|=\frac{n^{k}-n(n-1)\cdots(n-k+1)}{k!d^{k-1}}\leq \frac{kn^{k-1}}{d^{k-1}}\leq \frac{C}{n}
    \end{equation*}
and
\begin{align*}\label{G2}
    |\Gamma_{2}(W) - \Tilde{\Gamma}_{2}| &\leq \sum_{j=2}^{k-1} \binom{n}{k}\binom{k}{j} \left(\frac{n^{k-j}}{(k-j)!}-\binom{n-k}{k-j}\right) d^{1+j-2k} +k\binom{n}{k}\binom{n-k}{k-1}d^{2-2k}
    \notag \\
    &\quad + \binom{n}{k} \left[\binom{n-k}{k} - \binom{n}{k}\right] d^{2-2k} \notag \\
    &\leq C \bigg[\bigg(\sum_{j=2}^{k-1}n^{2k-j-1} d^{1+j-2k}\bigg) + n^{2k-1} d^{2-2k} \bigg] \notag \\
    &\leq \frac{C}{n}.
\end{align*}
For $|\Gamma_{3}(W)-\Tilde{\Gamma}_{3}|$, note that
$$m_{3}={\rm E}W^3=\sum_{\{i,j,k\}\in J^{3}}{\rm E}X_{i}X_{j}X_{k}.$$
Consider the following five cases separately. Let
\begin{align*}%\label{MM1}
    &J_{1}=\{\{i,j,m\}:\sharp\{i,j,m\}=1\},\notag\\
    &J_{2}=\{\{i,j,m\}:\sharp\{i,j,m\}=2\},\notag\\
    &J_{3}=\{\{i,j,m\}:\sharp\{i,j,m\}=3, \,\sharp(i\cup j \cup m)=3k\},\notag\\
    &J_{4}=\{\{i,j,m\}:\sharp\{i,j,m\}=3,\,i\cap (j\cup m)= \emptyset\;\text{or}\;(i\cup j)\cap m= \emptyset\;\text{or}\;j\cap (i\cup m)= \emptyset,\, \sharp(i\cup j \cup m)<3k\},\notag\\
    & J_{5}=J^{3}\backslash (J_{1}\cup J_{2}\cup J_{3}\cup J_{4}). \notag
\end{align*}
Simple calculations give
\begin{align*}%\label{MM2}
   & \sum_{\{i,j,m\}\in J_{1}}{\rm E}X_{i}X_{j}X_{m}={\rm E}W,\notag\\
   &   \sum_{\{i,j,m\}\in J_{2}}{\rm E}X_{i}X_{j}X_{m}=3({\rm E}W^{2}-{\rm E}W),\notag\\
    &\sum_{\{i,j,m\}\in J_{3}}{\rm E}X_{i}X_{j}X_{m}=({\rm E}W)^{3}+O \bigg(\frac{1}{n}\bigg ),\notag\\
      &\sum_{\{i,j,m\}\in J_{4}}{\rm E}X_{i}X_{j}X_{m}=3{\rm E}W^{2}{\rm E}W-3({\rm E}W)^{3}-3({\rm E}W)^{2}+O \bigg(\frac{1}{n}\bigg ),\notag\\
    &\sum_{\{i,j,m\}\in J_{5}}{\rm E}X_{i}X_{j}X_{m}=3\sum_{j=2}^{k-1}\binom{k}{j}\frac{a_jn^{2k-j}}{k!(k-j)!d^{2k-j-1}}+O \bigg(\frac{1}{n} \bigg).\notag
\end{align*}
Now, $\Gamma_{3}(W)$ can be calculated  by $ m_{3}-3m_{1}m_{2}-3m_{2}+2m_{1}^{3}+3m_{1}^{2}+2m_{1}$.  Direct but laborious computations yield
\begin{equation}\label{G3}
    |\Gamma_{3}(W)-\Tilde{\Gamma}_{3}|\leq Cn^{-1}. \notag
\end{equation}
The proof is now complete.
\end{proof}

\begin{theorem}\label{3.6}
    Suppose that both $n$ and $d$ tend  to $\infty$ such that $n^{k} \asymp d^{k-1}$. We have
     \begin{align}&\label{l371}
         d_\mathrm{TV}(W, M_3)\leq Cn^{-\frac{k}{k-1}},\\
  &\label{l372}
         d_{\rm loc}(W, M_3)\leq Cn^{-\frac{k}{k-1}},
    \end{align}
where  $\{\lambda,\omega,\eta\}$ is exactly determined by
\begin{equation}\label{112}
     \lambda+\omega+\eta=\Gamma_{1}(W),\quad \omega+2\eta=\Gamma_{2}(W),\quad 2\eta=\Gamma_{3}(W).
\end{equation}
\end{theorem}
\begin{proof}
    We only prove (\ref{l371}) since the proof of (\ref{l372}) is similar. We use $M_3$ to approximate $W$ since $\Gamma_2(W)\approx 0$  under the assumption $n^{k} \asymp d^{k-1}$. To apply Theorem \ref{main}(iii), we  need to control  $\gamma$, $S(W)$, and $\theta_3$.

Trivially,  $S(W)\leq 4$. Also, by Lemma \ref{L3.5}, noting that $\widetilde{\Gamma}_2/\widetilde{\Gamma}_1\rightarrow 0$ as $n, d\rightarrow\infty$, we know
$$
\theta_{3}=\frac{\Gamma_{2}(W)}{\Gamma_{1}(W)}\rightarrow 0.
$$
It remains to control $\gamma$. Note
\begin{equation}\label{cit}
    \sum_{j \in A_{i}\setminus \{i\}}{\rm E}X_{i}X_{j}=\sum_{j=1}^{k-1}\binom kj\binom{n-k}{ k-j}  d^{1+j-2 k}.
\end{equation}
One easily finds that for each $j\in A_{i}\setminus \{i\}$,
$
{\rm E}X_{i}{\rm E}X_{j} \leq {\rm E}X_{i}X_{j}.
$
According to  the argument in \cite{arratia1989two}, with $d/n$ sufficiently large, the dominant contribution to (\ref{cit}) comes
from the pairs $(i, j)$ with $\sharp(i\cap j)= k-1$,
which in turn implies
\[
\sum_{j\in A_{i}\setminus i}{\rm E}X_{i}X_{j}\leq Cnd^{-k}\leq Cd^{-k+1}.
\]
Observe
$$\{j,m: j \in A_{i}\setminus \{i\},m \in A_{ij}\setminus A_{i}\}=\{j,m:\sharp(i\cap j)>0,\sharp(i\cap m)=0, \sharp(j\cap m)>0\},$$ and thus
\begin{align}\label{3.111}
  \sum_{j \in A_{i}\setminus \{i\}}\sum_{m \in A_{ij}\setminus A_{i}}  {\rm E}X_{i}X_{j}X_{m} = \sum_{\{m:\sharp(i\cap m)=0\}}\sum_{\{j:\sharp(i\cap j)>0,\sharp(m\cap j)>0\}}{\rm E}X_{i}X_{j}X_{m}.
\end{align}
Fix $i$ and $m$ with $\sharp (i\cap m) =0$, and set $$F_{t}^{i,m} =\{j:\sharp(i\cap j)+\sharp(j\cap m)=k-t,\sharp(i\cap j)>0,\sharp(j\cap m)>0\}$$ with $t\geq 0$. The dominant contribution from $j$ to (\ref{3.111}) comes
from $F_{0}^{i,m}$. Indeed, by simple calculations,
${\rm E}X_{i}X_{j}X_{m}=d^{-2k-t+1}$ for $j\in F_{t}^{i,m}$, and $$\sharp F_{t}^{i,m}=\binom{n-2k}{t} \binom{2k}{k-t}, $$  which in turn implies
\begin{align*}
    \sum_{j \in A_{i}\setminus \{i\}}\sum_{m \in A_{ij}\setminus A_{i}}  {\rm E}X_{i}X_{j}X_{m} &=  \sum_{\{m:\sharp(i\cap m)=0\}}\sum_{t=0}^{k-2}\sum_{\{j:j\in F_{t}^{i,m}\}}{\rm E}X_{i}X_{j}X_{m}\\
    &\leq  Cn^{k}\sum_{t=0}^{k-2} n^{t}d^{-2k-t+1} \leq Cn^{k}d^{-2k+1}.
\end{align*}
Thus, (\ref{3.111}) is asymptotically  as large as $\sum_{\{m:\sharp(i\cap m)=0\}}\sum_{\{j:j\in F_{0}^{i,m}\}}{\rm E}X_{i}X_{j}X_{m}$.

For fixed $i$ and $m$ with $\sharp (i\cap m) =0$, set
 $$F_{t,s}^{i,m} =\{j,l:\sharp(i\cap j)+\sharp(j\cap m)=k-t,\,\sharp(i\cap j)>0,\,\sharp(j\cap m)>0,\,\sharp (l\cap (i\cup j \cup m))=k-s\}$$ with $t,s\geq 0$. By a similar argument, we claim  without proof that $F_{0,0}^{i,m}$ is the dominant contribution to
\begin{align}\label{3.12}
    \sum_{j \in A_{i}\backslash\{i\}}\sum_{m \in A_{ij}\backslash A_{i}}\sum_{l \in A_{ijm}\backslash \{m\}}  {\rm E}X_{i}X_{j}X_{m}X_{l}= \sum_{\{m:\,\sharp (i\cap m)=0)\}}\sum_{t=1}^{k-2}\sum_{s=1}^{k-1}\sum_{\{j,l:j,l\in F_{t,s}^{i,m}\}}  {\rm E}X_{i}X_{j}X_{m}X_{l}.
\end{align}
 Also,   ${\rm E}X_{i}X_{j}X_{m}X_{l}=d^{-2k+1} $ for $j\in F_{0,0}^{i,m}$, $\sharp F_{0,0}^{i,m}=\binom{2k}{ k}^{2}$, and  (\ref{3.12}) is,  up to a constant factor, as large as $n^{k}d^{-2k+1}$.

Finally, note
\begin{align*}\label{DE1}
    & \gamma^{1}=\sum_{i\in J}\sum_{j \in A_{i}\backslash\{i\}}  \sum_{m \in A_{ij}}\sum_{l \in A_{ijm}}  \sum_{({\rm E})}{\rm E}X_{i}({\rm E})X_{j}{\rm E}X_{m}({\rm E})X_{l},\notag\\
      & \gamma^{2}=\sum_{i\in J}\sum_{j \in A_{i}\backslash\{i\}}  \sum_{m \in A_{ij}\backslash A_{i}}\sum_{l \in A_{ijm}}  \sum_{({\rm E})}{\rm E}X_{i}({\rm E})X_{j}X_{m}({\rm E})X_{l}.\notag
\end{align*}
  Then,  it follows that
\begin{align*}
    \gamma^{1}&= \sum_{i\in J}\sum_{j \in A_{i}\setminus\{i\}}  \sum_{m \in A_{ij}}\sum_{l \in A_{ijm}}  \sum_{({\rm E})}{\rm E}X_{i}({\rm E})X_{j}{\rm E}X_{m}({\rm E})X_{l}\\
    &\leq \sum_{i\in J}\sum_{j \in A_{i}\setminus\{i\}}  \sum_{m \in A_{ij}}\sum_{l \in A_{ijm}\setminus{m}}  {\rm E}X_{i}X_{j}{\rm E}X_{m}X_{l}+\sum_{i\in J}\sum_{j \in A_{i}\setminus\{i\}}  \sum_{m \in A_{ij}}{\rm E}X_{i}X_{j}{\rm E}X_{m}\\
   &\leq   C \bigg(\frac{n^{2k+1}}{d^{2k}}\vee \frac{n^{2k}}{d^{2k-1}} \bigg )=Cn^{-\frac{k}{k-1}},\\
   \gamma^{2}&= \sum_{i\in J}\sum_{j \in A_{i}\setminus\{i\}}  \sum_{m \in A_{ij}\setminus A_{i}}\sum_{l \in A_{ijm}}  \sum_{({\rm E})}{\rm E}X_{i}({\rm E})X_{j}X_{m}({\rm E})X_{l}\\
    &\leq \sum_{i\in J}\sum_{j \in A_{i}\setminus\{i\}}  \sum_{m \in A_{ij}\setminus A_{i}}\sum_{l \in A_{ijm}\setminus m}  {\rm E}X_{i}X_{j}X_{m}X_{l}+\sum_{i\in J}\sum_{j \in A_{i}\setminus\{i\}}  \sum_{m \in A_{ij}\setminus A_{i}}  {\rm E}X_{i}X_{j}X_{m}\\
   &\leq   C \bigg(\frac{n^{2k}}{d^{2k-1}} \bigg)=Cn^{-\frac{k}{k-1}}.
\end{align*}
Hence, $\gamma=\gamma_1+\gamma_2\le Cn^{-\frac{k}{k-1}}$. We can now conclude the proof of (\ref{l371}).
\end{proof}
\begin{remark}$d_{\rm loc}$ is usually believed to be smaller than $d_\mathrm{TV}$. But as the reader might notice, the upper bounds in (\ref{l371}) and (\ref{l372}) are of the same order. This is mainly because ${\rm E}W$ and $\operatorname{Var}W$ are asymptotically constant such that   $|\Delta g_{A}|$ and $|\Delta g_{a}|$ are both $O(1)$, which in turn leads to the same upper bounds by   Stein's method.
\end{remark}
\begin{remark}
   Theorem \ref{T331} is a significant improvement compared with \cite{arratia1989two}, in which Arratia obtained
    $$d_\mathrm{TV} (W,\mathcal{P}(\Lambda) )\leq Cn^{-\frac{1}{k-1}},$$
    where $\Lambda=\binom n k d^{1-k},$
    and \cite{vatutin1982limit}, in which Vatutin and Mikhailov obtained
    $$d_{\rm loc} (W,Y^{d} )\leq C\sigma^{-1}=Cn^{-\frac{1}{k-1}}$$
    under the assumption $n^{k}\asymp d^{k-1}$.
\end{remark}

As the reader might realize, it is hard to obtain an explicit formula for the three parameters $\lambda$, $\omega$, and $\eta$ by solving the equations (\ref{112}) since $\Gamma_1(W)$, $\Gamma_2(W)$, and $\Gamma_3(W)$ are so complex. For the sake of practical computation, we   provide an alternative approximation using a new mixture of  Poisson variables. Define
\begin{align*}
    &\lambda^{\prime}=\frac{n^{k}}{k!d^{k}}+\sum_{j=2}^{k-1}\left(\frac{3}{2}a_j-1\right)\binom{k}{j}\frac{n^{2k-j}}{k!(k-j)!d^{2k-j-1}},\notag\\
    &\omega^{\prime}=\sum_{j=2}^{k-1}\left(1-3a_j\right)\binom{k}{j}\frac{n^{2k-j}}{k!(k-j)!d^{2k-j-1}},\notag\\
   &\eta^{\prime}=\frac{3}{2}\sum_{j=2}^{k-1}\binom{k}{j}\frac{a_jn^{2k-j}}{k!(k-j)!d^{2k-j-1}},
\end{align*}
and   $M_3'={\cal P}(\lambda') \ast 2{\cal P}(\omega'/2)\ast 3{\cal P}(\eta'/3)$.

\begin{lemma}\label{LL1}
    Assume $\theta_{3}<1/2$. We have
    \begin{equation*}
     d_\mathrm{TV}(M_3, M_3')\leq \frac{C}{n}.
 \end{equation*}
\end{lemma}
\begin{proof}
We follow the proof of  \cite[Corollary 2.4]{barbour1999poisson}, which deals with two-parameter cases.  Recall
\[
\mathcal{A}_{M_{3}}g(k)=(\lambda+\omega+\eta) g(k+1)-k g(k)+\omega\Delta g(k+1)+\eta[\Delta g(k+1)+\Delta g(k+2)]
\]
and
\[
\mathcal{A}_{M'_{3}}g(k)=(\lambda'+\omega'+\eta') g(k+1)-k g(k)+\omega'\Delta g(k+1)+\eta'[\Delta g(k+1)+\Delta g(k+2)].
\]
Thus, it follows that
\[
{\rm E}\mathcal{A}_{M'_{3}}g(M_3)=(\lambda'+\omega'+\eta') {\rm E}g(M_3+1)-{\rm E}M_3 g(M_3)+\omega' {\rm E}\Delta g(M_3+1)+\eta'[{\rm E}\Delta g(M_3+1)+{\rm E}\Delta g(M_3+2)]
\]
and
\[
 {\rm E}\mathcal{A}_{M_{3}}g(M_3)= (\lambda +\omega +\eta )  {\rm E}g(M_3+1)- {\rm E}M_3 g(M_3)+\omega  {\rm E}\Delta g(M_3+1)+\eta [ {\rm E}\Delta g(M_3+1)+ {\rm E}\Delta g(M_3+2)].
\]
Thus,   by the fact that ${\rm E}{\cal A}_{M_3}(M_3)=0$, we obtain
\begin{align}\label{eqq1}
   d_\mathrm{TV}(M_3, M_3')&= \sup_{A\subseteq \mathbb{Z}^+}|{\rm E}\mathcal{A}_{M'_{3}}g_A(M_3)|\notag\\&= \sup_{A\subseteq \mathbb{Z}^+}|{\rm E}\mathcal{A}_{M'_{3}}g_A(M_3)-{\rm E}\mathcal{A}_{M_{3}}g_A(M_3)| \nonumber\\
      &\leq  (|\lambda-\lambda'|+|\omega-\omega'|+|\eta-\eta'|)\sup_{A\subseteq \mathbb{Z}^+}|{\rm E}{g_{A}(M_3)}|\notag\\
   &\quad\, +(|\omega-\omega'|+2|\eta-\eta'|)\sup_{A\subseteq \mathbb{Z}^+}|{\rm E}{\Delta g_{A}(M_3)}|.
\end{align}
Hence, by \cite[Theorem 2.5]{barbour1999poisson} and Lemma \ref{s},
\[
 d_\mathrm{TV}(M_3, M_3') \le \frac{3(|\lambda-\lambda'|+|\omega-\omega'|+|\eta-\eta'|)}{(1-2\theta_{3})(\lambda+\omega+\eta)^{1/2}}.
\]

In addition, since
\begin{equation}\label{0i}
\lambda +\omega +\eta =\Tilde{\Gamma}_{1}(W),\quad \omega +2\eta =\Tilde{\Gamma}_{2}(W),\quad 2\eta =\Tilde{\Gamma}_{3}(W),
\end{equation}
and
\begin{equation*}\label{112222}
\lambda'+\omega'+\eta'=\Tilde{\Gamma}_{1},\quad \omega'+2\eta'=\Tilde{\Gamma}_{2},\quad 2\eta'=\Tilde{\Gamma}_{3},
\end{equation*}
then by Lemma  \ref{L3.5}, we have
 \begin{equation}\label{io}
    |\lambda-\lambda'|+|\omega-\omega'|+|\eta-\eta'|\leq \frac{C}{n}.
\end{equation}
Substituting (\ref{0i}) and (\ref{io}) into (\ref{eqq1}), we complete the proof.
\end{proof}

\begin{theorem}\label{TT3}
    Suppose that  both $n$ and $d$ tend  to $\infty$ such that $n^{k} \asymp d^{k-1}$. We have
    \begin{align}&\label{T331}
         d_\mathrm{TV}(W, M_3')\leq Cn^{-1},\\
    &\label{T332}
         d_{\rm loc}(W, M_3')\leq Cn^{-1}.
    \end{align}
\end{theorem}
\begin{proof}
Using Lemmas \ref{3.6} and \ref{LL1} and the triangle inequality, we get
\begin{align*}
   d_\mathrm{TV}(W, M_3')\leq  d_\mathrm{TV}(W, M_3) + d_\mathrm{TV}(M_3', M_3)
   \leq  Cn^{-\frac{k}{k-1}}+Cn^{-1}\leq \frac{C}{n}.
\end{align*}
We have proven that (\ref{T331}) holds. The proof of (\ref{T332}) is similar and  thus is omitted.
\end{proof}

\subsection{Counting monochromatic edges in uniformly colored graphs}
 Let $G = \{\mathcal{V} , \mathcal{E}\}$ be a simple  undirected graph, where $\mathcal{V}= \{v_{1}, \ldots, v_{n} \}$ is the vertex set and $\mathcal{E}= \{e_{1},e_{2},\ldots,e_{m_{n}}\}$  is the edge set. For $1 \leq i \leq n$, let $D_i$ denote the neighborhood of vertex $v_i$, i.e.,  $D_{i}= \{1 \leq j \leq n, j \neq i, (v_{i}, v_{j} ) \in \mathcal{E} \}$.  Set $d_{i}=\,\sharp D_{i}$, and note that  $2 m_{n}=\sum_{i=1}^{n} d_{i}$. Each vertex is colored independently and uniformly with $c_{n} \geq 2$ colors, and we are interested in the number of monochromatic edges in $\mathcal{E}$.

Define $J=\{1,2,\ldots,m_{n}\}$. For each edge $e_i\in \mathcal{E}$, we denote by $e_{i1}$ and $e_{i2}$ the two vertices it connects, i.e., $e_{i}= (e_{i1}, e_{i2})$.  Define \begin{equation*}
    X_{i}=\left\{\!\begin{array}{*{2}ll}
        1 &   \quad\text{if\ } e_{i1} \text{ and } e_{i2} \ \text{are colored the same}, \\
         0 &  \quad\text{otherwise}.
    \end{array}\right.
\end{equation*}
It can be easily verified that $\{X_i, i\in J\}$  satisfies the local dependence structure (LD1)--(LD3) with $A_{i}$, $A_{ij}$, and  $A_{ijk}$, where $A_{i}$ consists of all edges connected to $e_{i}$, $A_{ij}$  consists of all edges connected to $e_{i}$ and $e_{j}$  where $j \in A_{i}$, and $A_{ijk}$ consists of all edges connected to $e_{i}$, $e_{j}$, and $e_{k}$ where $j \in A_{i}$ and $k\in A_{j}$.

The number  of monochromatic edges in $G = \{\mathcal{V} , \mathcal{E}\}$ is defined as
  $$
  W=\sum_{j\in J} X_{j}.
  $$

In Figure \ref{FI13},  $n=7$, $c_{7}=3$, and $W =2$, where vertices $a$ and $f$ are both blue, while vertices $b$ and $e$ are  green.

\begin{figure}[h!]
\centering
    \begin{tikzpicture}
\fill [blue] (0,0) circle(1 pt)node[right]{\footnotesize$a$};
\fill [green] (0.5,-1) circle(1 pt)node[above,right]{\footnotesize$b$};
\fill [green] (-1,-1.2) circle(1 pt)node[left]{\footnotesize$e$};
\fill [blue] (0.2,-1.5) circle(1 pt)node[below,left]{\footnotesize$f$};
\fill [red] (0.2,-3) circle(1 pt)node[below]{\footnotesize$c$};
\fill [blue] (-2,-3.8) circle(1 pt)node[below,left]{\footnotesize$d$};
\fill [blue] (2,-0.5) circle(1 pt)node[above,right]{\footnotesize$g$};
\draw[line width=0.2 pt](0,0)--(0.5,-1) ;
\draw[line width=0.2pt](0.5,-1)--(0.2,-3) ;
\draw[line width=0.2pt](0.2,-3)--(-2,-3.8) ;
\draw[line width=0.2pt](-2,-3.8)--(-1,-1.2) ;
\draw[line width=0.2pt](-1,-1.2)--(0,0) ;
\draw[line width=0.2pt](0,0)--(0.2,-1.5) ;
\draw[line width=0.2pt](0.2,-1.5)--(0.2,-3) ;
\draw[line width=0.2pt](0.2,-1.5)--(-1,-1.2) ;
\draw[line width=0.2pt](-1,-1.2)--(0.5,-1) ;
\draw[line width=0.2pt](0.5,-1)--(0.2,-1.5) ;
\draw[line width=0.2pt](0.5,-1)--(2,-0.5) ;
\end{tikzpicture}
 \caption{(Color online) Uniformly colored graph with $n=7$, $c_{7}=3$, and $W=2$}
    \label{FI13}
\end{figure}

It is immediate that
\begin{align*}
&{\rm E}W=\frac{m_{n}}{c_{n}},\quad {\rm E}W^{2}=\frac{m_{n}(m_{n}-1)}{c_{n}^{2}}+\frac{m_{n}}{c_{n}},\\
&{\rm E}W^{3}=\frac{m_{n}(m_{n}-1)(m_{n}-2)}{c_{n}^{3}}
+3\frac{m_{n}(m_{n}-1)}{c_{n}^{2}}+\frac{m_{n}}{c_{n}}.\end{align*}
  Thus,  we have
\begin{equation*}\label{772}
   \Gamma_{1}(W)=\frac{m_{n}}{c_{n}},\quad\Gamma_{2}(W)=-\frac{m_{n}}{c_{n}^{2}},\quad \Gamma_{3}(W)= \frac{4m_{n}}{c_{n}^{3}}.
\end{equation*}
\begin{theorem}\label{T2.3} Define $d_{(n)}=\max_{1\le i\le n}d_{i}$, and assume $d_{(n)}^{2}\ll c_{n} \ll m_{n}$. Then
    \begin{align}\label{T361}
      & d_\mathrm{TV}(W,Y^{d}) \leq C \bigg(\sqrt{\frac{c_{n}}{m_{n}}}+\frac{d_{(n)}^{4}}{c_{n}^{3}} \bigg ),\\\label{T362}
      & d_{\rm loc}(W,Y^{d}) \leq C \bigg(\sqrt{\frac{c_{n}}{m_{n}}}+\frac{d_{(n)}^{4}}{c_{n}^{3}} \bigg)^{1/2} \sqrt{\frac{c_{n}}{m_{n}}+\frac{d_{(n)}}{\sqrt{c_{n}}}}.
 \end{align}
\end{theorem}
\begin{remark}
    Barbour et al. \cite{barbour1992poisson} used a single-parameter Poisson random variable to  approximate $W$ and obtained
 \begin{equation}\label{BP1}
     d_\mathrm{T V} \bigg(W, \mathcal{P} \bigg(\frac{m_{n}}{c_{n}} \bigg)\bigg )  \leq \frac{\sqrt{8 m_{n}}}{c_{n}}.
 \end{equation}
Obviously, (\ref{BP1}) makes sense only for $c_{n}\ll m_{n}\ll c_{n}^{2}$, and  the upper bound would increase with $m_{n}$.  In contrast, our result improves as $m_{n}$ grows larger.

We   mention some recent results
 due to \cite{fang2015universal,shao2019berry}, which provide upper bounds in terms of the Wasserstein distance and the Berry-Essen bound:
\begin{align*}
    &d_{W}(\mathcal{L}(W), N(\mu,\sigma^{2})) \leq 3 \sqrt{\frac{c_{n}}{m_{n}}}+\frac{10 \sqrt{2}}{\sqrt{c_{n}}}+\frac{1}{\sqrt{\pi}} \frac{2^{7 / 4}}{m_{n}^{1 / 4}},\\
&
    \sup _{z \in \mathbb{R}} |{\rm P} (W\leq \sigma z+\mu )-\Phi(z) | \leq C \bigg(\sqrt{\frac{1}{c_{n}}}+\sqrt{\frac{d_{(n)}}{m_{n}}}+\sqrt{\frac{c_{n}}{m_{n}}} \bigg ).
\end{align*}
By comparison, we find that the upper bounds in (\ref{T361}) and (\ref{T362}) are reasonable and correct,  and are indeed better when  $d_{(n)}$ is relatively small.
\end{remark}
\begin{proof}[Proof of Theorem $\ref{T2.3}$] We only prove (\ref{T361}) since the proof of (\ref{T362}) is similar.
Since $\Gamma_{2}\leq 0$, we prefer to use $M_{1}=B(n, p)*\mathcal{P}(\lambda)$ to approximate $W$. In our context, the three parameters become $n=m_{n}$, $p=c_{n}^{-1}$, and $\lambda=0$, i.e., $M_1=B(m_n, c_{n}^{-1})$.  To apply
 Theorem \ref{main}(i), we need to control $\gamma$, $S(W)$, and $\theta_1$ separately.

First, note that for each $ i \in J$,
\begin{equation}\label{TG1}
{\rm E}X_{i}={\rm E}X_{i}^{2}=\frac{1}{c_{n}}
\end{equation}
and
\begin{equation}\label{TG2}
 \sum_{j\in A_{i}\backslash \{i\}}{\rm E}X_{i}X_{j}\le \frac{d_{(n)}}{c_{n}^{2}}.
\end{equation}
 Then,  by (\ref{TG1}), (\ref{TG2}),  and the assumption $d_{(n)}\leq \sqrt{c_{n}}$,
\begin{align}
   \gamma^{1}&:=\sum_{i\in J}\sum_{j \in A_{i}\backslash\{i\}}  \sum_{k \in A_{ij}}\sum_{l \in A_{ijk}}  \sum_{({\rm E})}{\rm E}X_{i}({\rm E})X_{j}{\rm E}X_{k}({\rm E})X_{l}\notag\\
   &\ \leq  C \sum_{i\in J}\sum_{j\in A_{i}\backslash \{i\}}{\rm E}X_{i}X_{j}\sum_{k\in A_{ij}}{\rm E}X_{k}^{2}\notag\\
   &\ \leq  C\frac{m_{n}d_{(n)}^{2}}{c_{n}^{3}},\notag\\
   \gamma^{2}&:=\sum_{i\in J}\sum_{j \in A_{i}\backslash\{i\}}  \sum_{k \in A_{ij}\backslash A_{i}}\sum_{l \in A_{ijk}}  \sum_{({\rm E})}{\rm E}X_{i}({\rm E})X_{j}X_{k}({\rm E})X_{l}\notag\\
   & \ \leq  C \sum_{i\in J}\sum_{j\in A_{i}\backslash \{i\}}{\rm E}X_{i}X_{j}\sum_{k\in A_{ij}\backslash A_{i}}X_{k}^{2}\notag\\
   &\ \leq  C\frac{m_{n}d_{(n)}^{2}}{c_{n}^{3}}.\notag
\end{align}
 Thus,  we have
 \begin{align}\label{341}
    \gamma &=  \gamma^{1} + \gamma^{2} \leq C\frac{m_{n}d_{(n)}^{2}}{c_{n}^{3}}.
\end{align}

Second, note $\lambda=0$ and thus $\theta_{1}=0$.

It remains to calculate $S(W)$. By
\cite[Proposition 4.6]{barbour1999poisson}, we easily get
\begin{equation}\label{343}
    S_{2}(M_{1})\leq C\frac{c_{n}}{m_{n}}.
\end{equation}
Define $p_{k}={\rm P}(W=k)$ and $q_{k}={\rm P}(M_{1}=k)$, $k\in \mathbb{Z}_+$.  It follows that
\begin{align}\label{344}
    |S_{2}(W)-S_{2}(M_{1})|&\leq  \bigg|\sum_{k=0}^{+\infty}[|p_{k+2}-2p_{k+1}
    +p_{k}|-|q_{k+2}-2q_{k+1}+q_{k}|]\bigg|\notag\\
    &\leq  4\sum_{k=0}^{+\infty}|p_{k}-q_{k}| = 4d_\mathrm{TV}(W, M_{1}).
    \end{align}

In addition,  we have for each configuration $x$,
\begin{equation}\label{Claim1}
       S_{2}(W| X_{A_{ijk}}=x)-S_{2}(W)\leq C\frac{d_{(n)}^{2}}{c_{n}},
\end{equation}
whose proof is postponed to \ref{A1}.

Combining (\ref{343})--(\ref{Claim1}) directly implies
\begin{align}\label{3441}
S(W)\le  S_{2}(W)+ C\frac{d_{(n)}^{2}}{c_{n}}
\le    4d_\mathrm{TV}(W, M_{1})+ C\frac{c_{n}}{m_{n}} + C\frac{d_{(n)}^{2}}{c_{n}}.
  \end{align}
Finally,  by  Theorem \ref{main}(i),   (\ref{341}), and (\ref{3441}), noting that $\theta_1=0$ and $\mu=m_n/c_n$, we have
\begin{align} \label{3442}
    d_\mathrm{TV}(W,M_{1}) \leq  C \frac{\gamma S(W)+p^2}{(1-2\theta_1)q\mu}
      \leq  C\frac{d_{(n)}^{2}}{c_{n}^{2}}\bigg[\frac{c_{n}}{m_{n}}+\frac{d_{(n)}^{2}}{c_{n}}+4d_\mathrm{TV}(W, M_{1})\bigg].
\end{align}

In turn, since $d_{(n)}^{2}\ll c_n$, solving (\ref{3442}) yields
\begin{equation}\label{22}
    d_\mathrm{TV}(W, M_{1}) \leq C \bigg(\frac{d_{(n)}^{2}}{m_{n}c_{n}}+\frac{d_{(n)}^{4}}{c_{n}^{3}}+\frac{c_{n}}{m_{n}} \bigg ).
\end{equation}
On the other hand,  it follows from Lemma \ref{L4.1} that
$$d_\mathrm{TV}(M_{2}, Y^{d})\leq \frac{C}{\sqrt{\operatorname{Var}W}}= C\sqrt{\frac{c_{n}}{m_{n}}},$$
which, together with (\ref{22}), implies
 \begin{align*}
      d_\mathrm{TV} (W, Y^{d} ) \leq C \bigg(\sqrt{\frac{c_{n}}{m_{n}}}+\frac{d_{(n)}^{2}}{m_{n}c_{n}}+\frac{d_{(n)}^{4}}{c_{n}^{3}} \bigg)\leq C \bigg(\sqrt{\frac{c_{n}}{m_{n}}}+\frac{d_{(n)}^{4}}{c_{n}^{3}}\bigg ),
 \end{align*}
 where  in the last inequality we used the fact that $$\frac{d_{(n)}^{2}}{m_{n}c_{n}}\leq \sqrt{\frac{c_{n}}{m_{n}}}.$$

The proof is now complete assuming  (\ref{Claim1}).
\end{proof}
\subsection{Triangles in the Erd\H{o}s-R\'{e}nyi random graph}
In this subsection, we consider the number of triangles in the classical Erd\H{o}s-R\'{e}nyi random graph. Specifically, let $G = G(n, p)$ be a random
graph with $n$ vertices, where each edge appears with probability $p$, independent of all
other edges.
Set $J=\{1,2,\ldots,\binom{n}{3}\}$, and denote by $\{T_{i}, i\in J\}$ all possible triangles between these $n$ vertices. For each $i\in J$, define
$$X_{i}=
        \left\{\!
             \begin{array}{ll}
             1&\quad\text{if} \;\;T_{i} \;\;\text{exists in } G(n, p),\\
             0&\quad\text{otherwise}.
             \end{array} \right.
$$
Set, for each $i\in J$,
$$A_{i}= \{j \in J: e (T_{j} \cap T_{i} ) \geq 1 \},$$
 for each $j\in A_{i}$,
  $$A_{i j}= \{k \in J: e (T_{k} \cap (T_{i} \cup T_{j} ) ) \geq 1 \},$$
 and for each $k\in A_{ij}$,
   $$A_{i j k}= \{l \in J: e (T_{l} \cap (T_{i} \cup T_{j} \cup T_{k} ) ) \geq 1 \}.$$
We can easily verify  that $\{X_{i}, i\in J\}$ satisfies the local dependence structure (LD1)--(LD3) with  $A_{i}$, $A_{ij}$ and $A_{ijk}$.

Denote by $W$   the number of triangles in  $G(n, p)$, i.e., $W=\sum_{i \in J} X_{i}$.  Figure 3 illustrates  $G(10, 0.2)$, where $\{v_{1}, v_{4}, v_{5}\}$ and $\{v_{2}, v_{5}, v_{7}\}$ each form a triangle. Thus, $W=2$.

It is immediate that
    \begin{equation*}
         {\rm E}W=\binom{n}{3}p^{3},\quad \operatorname{Var}W= \binom{n}{3}p^{3}+\binom{n}{3}(3n-9)p^{5}-\binom{n}{3}(3n-8)p^{6}
    \end{equation*}
and
\begin{align}\label{T4G1}
        \Gamma_{1}=\binom{n}{3}p^{3},\quad  \Gamma_{2}=\binom{n}{3}(3n-9)p^{5}-\binom{n}{3}(3n-8)p^{6},\quad \Gamma_{3}=n^{5}p^{7}+o(n^{5}p^{7}).
    \end{align}

Our main result about $W$ reads as follows.
\begin{theorem}\label{T2.4}
 Suppose $n^{\alpha}p \rightarrow c > 0$ with $1/2 \leq \alpha < 1$.
  Then we have
\begin{align}\label{381}
     &d_\mathrm{TV}(W,M_{2})\leq n^{-1-\alpha},\quad  d_\mathrm{TV}(W,Y^{d})\leq n^{-3/2+3\alpha/2},\\
     &\label{382}
  d_{\rm loc}(W,M_{2})\leq n^{-2+\alpha},\quad d_{\rm loc}(W,Y^{d})\leq n^{-3+3\alpha}.
   \end{align}
\end{theorem}
\begin{remark}
    R\"{o}llin and Ross \cite{rollin2015local} studied  the total variation distance  between $W$ and a  translated Poisson $Z$, where $Z- \lfloor\mu-\sigma^{2}\rfloor \sim \operatorname{\cal P} (\sigma^{2}+\gamma)$, and obtained
 \begin{align}\label{ER1}
 d_\mathrm{TV}(W,Z)\leq O(n^{-1+\alpha}),
 \end{align}
where $\mu = {\rm E}W$, $\sigma^{2}= \operatorname{Var}W$, and $\gamma=\mu-\sigma^{2}- \lfloor\mu-\sigma^{2} \rfloor$.

 \begin{figure}[h!]
\centering
    \begin{tikzpicture}
\fill [black] (2.5,0.5) circle(1 pt)node[right]{\footnotesize$1$};
\fill [black] (2,1) circle(1 pt)node[above,right]{\footnotesize$2$};
\fill [black] (1,2) circle(1 pt)node[left,right]{\footnotesize$3$};
\fill [black] (0,2) circle(1 pt)node[below,left]{\footnotesize$4$};
\fill [black] (-1,1) circle(1 pt)node[above,left]{\footnotesize$5$};
\fill [black] (-1.5,0.5) circle(1 pt)node[above,left]{\footnotesize$6$};
\fill [black] (-1,0) circle(1 pt)node[below,left]{\footnotesize$7$};
\fill [black] (0,-1) circle(1 pt)node[below,left]{\footnotesize$8$};
\fill [black] (1,-1) circle(1 pt)node[below,right]{\footnotesize$9$};
\fill [black] (2,0) circle(1 pt)node[above,right]{\footnotesize$10$};

\draw[line width=0.2 pt,red](2.5,0.5)--(0,2) ;
\draw[line width=0.2pt,red](2.5,0.5)--(-1,1) ;
\draw[line width=0.2pt,red](2,1)--(-1,1) ;
\draw[line width=0.2pt,red](2,1)--(-1,0) ;
\draw[line width=0.2pt,blue](2,1)--(0,-1) ;
\draw[line width=0.2pt,blue](1,2)--(-1,1) ;
\draw[line width=0.2pt,blue](1,2)--(1,-1) ;
\draw[line width=0.2pt,red](0,2)--(-1,1);
\draw[line width=0.2pt,red](-1,1)--(-1,0) ;
\draw[line width=0.2pt](-1,0)--(2,0) ;
\end{tikzpicture}
  \caption{(Color online) Erd\"{o}s-R\'{e}nyi random
graph with $n=10$ and $p=0.2$ }
    \label{FI15}
\end{figure}

Our result  in (\ref{381}), either $d_\mathrm{TV}(W,M_{2}) \leq n^{-1-\alpha}$ or  $d_\mathrm{TV}(W,Y^{d})\leq n^{-3/2+3\alpha/2}$,  offers higher accuracy compared with (\ref{ER1}). In addition, for any $1/2 \leq \alpha < 1$, the approximation accuracy by $M_{2}$ is better than that by~$Y^{d}$. \\ \indent
There are several results in the literature regarding the normal approximation of $W$ in terms of the Wasserstein-1 distance and the Kolmogorov distance. For example, Barbour et al.~\cite{barbour1989central} gave
 $$d_{W_{1}}(W,Y) \leq C\left\{\!\begin{array}{ll}n^{-\frac{3}{2}} p^{-\frac{3}{2}} & \ \text { if } 0<p \leq n^{-\frac{1}{2}}, \\ n^{-1} p^{-\frac{1}{2}} &\  \text { if } n^{-\frac{1}{2}}<p \leq \dfrac{1}{2}, \\[2ex]
  n^{-1}(1-p)^{-\frac{1}{2}} & \ \text { if } \dfrac{1}{2}<p<1,\end{array}\right.$$
and R\"{o}llin \cite{rollin2022kolmogorov} showed that  for every $n \geq 3$ and every $0 < p < 1$,
     $$d_{\mathrm{K}}(W,Y) \leq C\left\{\begin{array}{ll}n^{-\frac{3}{2}} p^{-\frac{3}{2}} &\  \text { if } 0<p \leq n^{-\frac{1}{2}}, \\ n^{-1} p^{-\frac{1}{2}} &\  \text { if } n^{-\frac{1}{2}}<p \leq \dfrac{1}{2}, \\[2ex]
      n^{-1}(1-p)^{-\frac{1}{2}} &\  \text { if } \dfrac{1}{2}<p<1,\end{array}\right.$$
     where $Y \sim \mathcal{N}(\mu,\sigma^{2})$.\\\indent
We note that our upper bound (\ref{381}) for the total variation distance achieves comparable accuracy for $ n^{-1} \leq p \leq n^{-1/2} $. However, for $ p > n^{-1/2} $, the corresponding upper bound is less satisfactory. In fact, R\"{o}llin acknowledged this issue in \cite{rollin2015local}, where he mentioned: ``it is not clear if this is an artifact of our method or if a standard local limit theorem does not hold if $ p > n^{-1/2} $."
\end{remark}

Our method can also be extended to other subgraph counting problems. A key technical challenge is to estimate $ S(W) $ and $ {\rm E}W^3 $, which we leave to the interested reader.

\begin{remark}
We believe that the error bound in Theorem \ref{T2.4} is nearly optimal for such problems. Table 1 contains numerical experiments that support this belief.

Additionally, Figure 4 shows that the $ R $-squared value is very close to 1, indicating a good fit and further confirming that the order of approximation is consistent with our theoretical results.
\end{remark}

\begin{proof}[Proof of Theorem $\ref{T2.4}$] We only prove \eqref{381}  since  combining $\eqref{381}$,
\cite[Theorem 2.2]{rollin2015local}, and Corollary \ref{C21} is sufficient for proving \eqref{382}.  It follows from (\ref{T4G1}) that $$\Gamma_{1}(W)\sim n^{3}p^{3}/6,\quad \Gamma_{2}(W)\sim n^{4}p^{5}/2,\quad\Gamma_{3}(W)\sim n^{5}p^{7}.$$

\begin{table}[h!]\footnotesize\tabcolsep 16pt
\begin{center}
\caption{Error in  the normal approximation of triangle counting}\label{Ta1}
\end{center}
\begin{center}
\begin{tabular}{ccccc}
\toprule
\diagbox{$p$}{$N$} & $300$ & $400$ & $500$ & $600$  \\
\hline
$N^{-0.6}$  &  $0.04600$   &  $0.03589$  & $0.03360$    &   $0.03050$  \\
%\hline
$N^{-0.7}$& $ 0.06490$  & $0.05900$   & $0.05100$   &  $0.04727$   \\
%\hline
$N^{-0.8}$  &  $0.11980$   &  $0.11120$   &   $0.10740$  &   $0.10150$   \\
\bottomrule
\end{tabular}
\end{center}
\vspace{4mm}\end{table}

\begin{figure}[h!]\footnotesize
  \centering
  \begin{tikzpicture}
    \begin{axis}[
        xlabel={$x$ in $N^{-3/2}p^{-3/2}$},
        ylabel={$y$ in $d_\mathrm{TV}(W,Y)$},
        width=4in,
        height=2.5in,
        legend pos=north west,
        legend style={},
        xticklabel style={/pgf/number format/.cd,
            fixed,
            fixed zerofill,
            precision=2,
        /tikz/.cd},
        yticklabel style={/pgf/number format/.cd,
            fixed,
            fixed zerofill,
            precision=2,
        /tikz/.cd},
        title style={font=\footnotesize},  % Match caption font size
        label style={font=\footnotesize},  % Match caption font size
    ]
      \addplot table {
        X       Y
        0.1806  0.1198
        0.1657  0.1112
        0.155   0.1074
        0.1467  0.1015
        0.0768  0.0649
        0.0675  0.059
        0.061   0.051
        0.0562  0.04727
        0.0326  0.046
        0.0275  0.03589
        0.024   0.0336
        0.0215  0.0305
      };

      \addplot table {
        X       Y
        0.1806  0.12958
        0.1657  0.1188898
        0.155   0.1112125
        0.1467  0.10152874
        0.0768  0.055096
        0.0675  0.0484055
        0.061   0.043782
        0.0562  0.0403304
        0.0326  0.0234184
        0.0275  0.019837
        0.024   0.0172365
        0.0215  0.01545
      };

      \addlegendentry{\footnotesize\hspace{-10mm}$y(x)$}
      \addlegendentry{\footnotesize$y=0.7175x$}
    \end{axis}
  \end{tikzpicture}
  \label{F154}
  \caption{(Color online) Linear function fitting graph of $y(x)$ with $R^{2}=0.9736$}
\end{figure}

\noindent
Since $\Gamma_{2}>0$, we choose $M_{2}=NB(r,\bar{p})*\mathcal{P}(\lambda)$ to approximate $W$,  where $r, \bar{p}$, and $\lambda$ are determined by~(\ref{k2}). Some simple calculations show
  \begin{align*}%\label{sat}
r=O(n^{2}p),  \quad \bar{p} =1- O(np^2),  \quad
\bar{q}=1-\bar{p}=O(np^{2}),\quad \lambda = O(n^{3}p^{3}) .\notag
\end{align*}
We control $\gamma$, $\theta_{2}$, and $S(W)$ separately.
Note that $np^{2} \rightarrow 0$, and then
$${\rm E}X_{i}={\rm E}X_{i}^{2}=p^{3}\geq Cnp^{5}\geq \sum_{j \in A_{i}\backslash \{i\}}{\rm E}X_{i}X_{j}.$$
Thus, we have
\begin{align}
   &\gamma^{1}:= \sum_{i\in J}\sum_{j \in A_{i}\setminus\{i\}}  \sum_{k \in A_{ij}}\sum_{l \in A_{ijk}}  \sum_{({\rm E})}{\rm E}X_{i}({\rm E})X_{j}{\rm E}X_{k}({\rm E})X_{l}
   \notag\\
   &\quad\  \leq  \sum_{i\in J} \sum_{j \in A_{i}\setminus \{i\}}\sum_{k\in A_{ij}}{\rm E}X_{i}X_{j}{\rm E}X_{k}^{2} \notag\\&\quad\leq Cn^{5}p^{8},\notag\\
 & \gamma^{2}:= \sum_{i\in J}\sum_{j \in A_{i}\backslash\{i\}}  \sum_{k \in A_{ij}\backslash A_{i}}\sum_{l \in A_{ijk}}  \sum_{({\rm E})}{\rm E}X_{i}({\rm E})X_{j}X_{k}({\rm E})X_{l}
   \notag\\
   &\quad\   \leq
   \sum_{i\in J} \sum_{j \in A_{i}\setminus \{i\}}\sum_{k\in A_{ij}\setminus A_{i}}{\rm E}X_{i}X_{j}X_{k}^{2} \notag\\&\quad\leq Cn^{5}p^{7},\notag
 \end{align}
 from which it follows that
   \begin{align}\label{TN1}
   \gamma=\gamma^{1}+\gamma^{2}\leq Cn^{5}p^{7}.
\end{align}

Turning to $\theta_{2}$, it  is easy to see  from Lemma \ref{s} that
\begin{align}\label{TN2}
    & \theta_{2} = \frac{\lambda q}{\lambda p+rq}\rightarrow 0.
\end{align}

Finally, we have
\begin{equation}\label{318}
   S(W)= O(n^{-3}p^{-3}),
\end{equation}
whose proof follows from \cite[Lemma 4.10]{rollin2015local} and is postponed to  \ref{BB}.

Substituting (\ref{TN1})--(\ref{318}) into  Theorem \ref{main}(ii) and using Corollary   \ref{main2}, we complete the proof of (\ref{381}).
\end{proof}

%%%%%%%%%%%%%%%%%%%%%%%%%%%%%%%%%%%%%%%%%%%%%%%%%%%%%%%%%%%%%%%%%%%%%%

%%%%%%%%%%%%%%%%%%%%%%%%%%%%%%%%%%%%%%%%%%%%%%%%%%%%%%%
%%% Acknowledgements. ??D?
%%%%%%%%%%%%%%%%%%%%%%%%%%%%%%%%%%%%%%%%%%%%%%%%%%%%%%%
\textbf{Acknowledgements:} This work was   supported by National Natural Science Foundation of China (Grant Nos. 12271475 and U23A2064).
The authors  express their sincere gratitude to referees for the careful readings and constructive comments.

%%%%%%%%%%%%%%%%%%%%%%%%%%%%%%%%%%%%%%%%%%%%%%%%%%%%%%%
%%% Conflict of interest. ¡Á¡Â??¨¤?¨°?¨¦¨´?¡Â
%%%%%%%%%%%%%%%%%%%%%%%%%%%%%%%%%%%%%%%%%%%%%%%%%%%%%%%
%\InterestConflict

%%%%%%%%%%%%%%%%%%%%%%%%%%%%%%%%%%%%%%%%%%%%%%%%%%%%%%%
%%% Supplements. 213?2?¨¢?, ¡¤?¡À???
%%%%%%%%%%%%%%%%%%%%%%%%%%%%%%%%%%%%%%%%%%%%%%%%%%%%%%%
%\Supplements{}

%%%%%%%%%%%%%%%%%%%%%%%%%%%%%%%%%%%%%%%%%%%%%%%%%%%%%%%
%%% Reference section. 2??????¡Á
%%% citation in the content using "some words~\cite{1,2}".
%%% ~ is needed to make the reference number is on the same line with the word before it.
%%%%%%%%%%%%%%%%%%%%%%%%%%%%%%%%%%%%%%%%%%%%%%%%%%%%%%%

%%%%%%%%%%%%%%%%%%%%%%%%%%%%%%%%%%%%%%%%%%%%%%%%%%%%%%%
%%% Appendix sections. ???????¨², ¡¤?¡À???
%%%%%%%%%%%%%%%%%%%%%%%%%%%%%%%%%%%%%%%%%%%%%%%%%%%%%%%
\begin{appendix}

\section{Proof of (\ref{Claim1})}\label{A1}
Call an edge \textit{undetermined}  if its vertices are colored independently and uniformly with $c_{n}$ colors. Let   $A$ denote the event that each of the $l$  \textit{undetermined} edges has vertices of different colors.
We now present a lemma that plays a crucial role in estimating the probability ${\rm P}(A)$.
\begin{lemma}\label{LA}
It holds that
 $1-\frac {l}{c_n}\leq {\rm P}(A) <1.$
\end{lemma}
\begin{proof}
    The upper bound is trivial. We now focus on the lower bound. Assume that the $l$ edges contain $m$ vertices $\{1,2,\ldots,m\}$. Denote by $r^{i}_{j,k,\ldots,l}$  the total number of edges between $\{i\}$ and $\{j,k,\ldots, l\}$, and by $R^{i}_{j,k,\ldots, l}$ the indicator that the two vertices of each edge between  $\{i\}$ and $\{j,k,\ldots,l\}$ have different colors. Then,      \begin{align}\label{AA1}
        {\rm P}(A)&= {\rm P}(R^{1}_{2}R^{3}_{1,2}\cdots R^{m}_{1,2,\ldots,m-1}=1)\notag\\
        &= {\rm P}(R^{1}_{2}=1){\rm P}(R^{3}_{1,2}=1| R^{1}_{2}=1)\cdots {\rm P}(R^{m}_{1,2,\ldots,m-1}=1| R^{m-1}_{1,2,\ldots,m-2}=1).
    \end{align}
It is immediate that for each $1\leq i \leq m$,
\begin{align}\label{AA2}
    {\rm P}(R^{i}_{1,2,\ldots,i-1}=1| R^{i-1}_{1,2,\ldots,i-2}=1)\geq \bigg(1-\frac{r^{i}_{j,k,\ldots,i}}{c_{n}}\bigg )
\end{align}
and
\begin{align}\label{AA3}
    \sum_{i=2}^{m}r^{i}_{j,k,\ldots,i}=l.
\end{align}
Substituting (\ref{AA2}) and (\ref{AA3}) into (\ref{AA1}), we get
\begin{align}
        {\rm P}(A)\geq \bigg(1-\frac{r^{1}_{2}}{c_{n}}\bigg ) \bigg (1-\frac{r^{3}_{1,2}}{c_{n}}\bigg )\cdots \bigg (1-\frac{r^{m}_{1,2,\ldots,m-1}}{c_{n}}\bigg )\geq  1-\frac{l}{c_{n}}.
    \end{align}
The proof is complete.
\end{proof}

 Set $m=\,\sharp \{l\in A_{ijk}, X_{l}=1\}$, and  let $W_{x}$  denote a random variable distributed as $\mathcal{L}(W|X_{A_{ijk}} = x)$. It is immediate that
    \begin{align}\label{C32}
    S_{2}(W|X_{A_{ijk} }=x)&= \sum_{k=0}^{m_{n}} | {\rm P}(W_{x}=m+k)-2{\rm P}(W_{x}=m+k-1)+{\rm P}(W_{x}=m+k-2) |\notag\\
    &\leq \sum_{k=0}^{m_{n}} | {\rm P}(W=k)-2{\rm P}(W=k-1)+{\rm P}(W=k-2)  |\notag\\
    &\quad\,  +4\sum_{k=0}^{m_{n}} | {\rm P}(W_{x}=m+k)-{\rm P}(W=k) |\notag\\
    &=   S_{2}(W)+4\sum_{k=0}^{m_{n}}| {\rm P}(W_{x}=m+k)-{\rm P}(W=k) |.
\end{align}

It remains to estimate the upper bound for the second term on the right-hand side  of (\ref{C32}). To do this, we introduce some additional notations. Divide the index set $J$ into two disjoint subsets, denoted by $J_{1}$ and  $J_{2}$, where  $J_1$ consists of $l\in J$ such that the edge $e_{l} $  has  no common vertex with the edges from $\{e_{l}, l\in A_{ijk}\}$. Define $W_{1}=\sum_{i\in J_{1}}X_{i}$ and $W_{2}=\sum_{i\in J_{2}}X_{i}$. Obviously, $W=W_{1}+W_{2}$.
For simplicity, set
\begin{align*}
  &  W_{2,l}\xlongequal{d}W_{2}| W_{1}=l,\quad
 W_{x} \xlongequal{d} W| X_{A_{ijk}} = x,\\
& W_{x,1}\xlongequal{d}W_{1}| X_{A_{ijk}} = x,\quad  W_{x,2,l}\xlongequal{d}W_{2}| X_{A_{ijk}} = x, W_{1}=l.
\end{align*}
We can easily derive that for every $k \ge 0$,
\begin{align*}
    {\rm P}(W=k)=\sum_{i=0}^{k}{\rm P}(W_{1}=k-i){\rm P}
(W_{2,k-i}=i)
\end{align*}
and
\begin{align*}
{\rm P}(W_{x}=k+m)=\sum_{i=0}^{k}{\rm P}(W_{x,1}=k-i){\rm P}
(W_{x,2,k-i}=m+i).
\end{align*}
Note that $W_{1}$ is independent of $X_{A_{ijk}}$. Thus, $W_{x,1}\xlongequal{d} W_{1}$. We have
\begin{align}\label{SUB2}
   {\rm P}(W=k)- {\rm P}(W=k+m | x_{A_{ijk}}) = \sum_{i=0}^{k}{\rm P}(W_{1}=k-i) [{\rm P}
(W_{2,k-i}=i)-{\rm P}
(W_{x,2,k-i}=m+i) ].
\end{align}
\begin{lemma}\label{LB11}
    We have
    \begin{equation*}\label{W2k}
    {\rm P}(W_{2,k}=0) \geq 1-\frac{Cd_{(n)}^{2}}{c_{n}}
\end{equation*}
and
\begin{equation*}\label{Wxk}
    {\rm P}(W_{x,2,k}=m) \geq 1-\frac{Cd_{(n)}^{2}}{c_{n}}.
\end{equation*}
\end{lemma}
\begin{proof}Introduce the following subsets:
\begin{align*}
    &V_{1} = \{ v \in \mathcal{V} : \exists\, l \in A_{ijk} \text{ such that } v \text{ is a vertex of } e_{l} \},\\
&V_{2} = \{ v \in \mathcal{V} : \exists\, l \in J_2\, \backslash A_{ijk} \text{ such that } v \text{ is a vertex of } e_{l} \},\\
& U_{1}=\{e:e=\{v_{s},v_{t}\}, v_{s}\in V_{1}, v_{t} \notin  V_{1}\},\\
& U_{2}=\{e:e=\{v_{s},v_{t}\}, v_{s}\in V_{1}, v_{t} \in  V_{1}\},\\
& V=\{  v\in \mathcal{V}:\,   \text{the color of $v$ is different from that of its neighborhoods}\},\\
& U=\{  u\in \mathcal{E}:\,   \text{the two vertices of $u$ have different colors}\},\\
& V_3=V_1\,\backslash V_2,\quad V_{1,2}=V_{1}\cap V_{2},\quad U^{1,2}=U_{1}\cup U_{2}.
\end{align*}
Some calculus  gives
$$\sharp V_{1} =O(d_{(n)}), \quad   \sharp  V_{1,3} =O(d_{(n)})  $$
and
$$ \sharp  U_{1} =O(d_{(n)}^{2}), \quad \sharp  U_{2}  =O(d_{(n)}^{2}),  \quad \sharp U^{1,2} =O(d_{(n)}^{2}).$$
Applying Lemma \ref{LA} to the event $\{X_{l}=0,l\in U_{2}\}$ yields
\begin{equation}\label{V123}
    1-\frac{\sharp U_{2} }{c_{n}}\leq {\rm P}(X_{l}=0,l\in U_{2})<1.
\end{equation}
Also,  note that
$$\{ X_{l}=0,l\in U_{1,2}\} = \{\{v_{s},v_{t}\}\in U \;\text{for all}\;v_{s}\in V_{1,2},v_t\in V_{3}; X_{l}=0,l\in U_{2}\}$$
and $$ \{\{v_{s},v_{t}\}\in U \;\text{for all}\;v_{s}\in V_{1,2},v_t\in V_{3}\}\supset \{v\in V \;\text{for all}\;v\in V_{1,2}\}.$$

Since each vertex in $V_{1,2}$ is colored independently,  we have
\begin{align*}
    &{\rm P} (v\in V \;\text{for all}\;v\in V_{1,2}\mid X_{l}=0,l\in U_{2},W_1=k )\notag\\
    &\quad\geq {\rm P} (v\in V \;\text{for all}\;v\in V_{1,2}\mid \;\text{the neighborhoods of each}\; v\in V_{1,2}\; \text{have different colors} )\notag\\
    &\quad\geq  \bigg(1-\frac{d_{(n)}}{c_{n}}\bigg)^{\sharp V_{1,2} }.
\end{align*}
Also, note that the restriction $\{W_{1}=k\}$  only affects the color of vertices in $V_{1,2}$. Therefore, it follows that
\begin{align*}
    {\rm P}(W_{2,k}=0)&= {\rm P}(X_{l}=0,l\in U^{1,2}\mid W_1=k)\notag\\
    &= {\rm P}(X_{l}=0,l\in U_{2}; \{v_{s},v_{t}\}\in U \;\text{for all}\;v_{s}\in V_{1,2},v_t\in V_{3}\mid W_1=k)\notag\\
    &\geq {\rm P}(X_{l}=0,l\in U_{2}){\rm P}(v\in V \;\text{for all}\;v\in V_{1,2} \mid X_{l}=0,l\in U_{2},W_1=k)\notag\\
    &\geq    \bigg(1- \frac{\sharp U_{2} }{ c_{n}}\bigg)\bigg(1-\frac{d_{(n)}}{c_{n}}\bigg)^{\sharp V_{1,2} }.
\end{align*}
Since $\sharp U_{2}=O(d_{(n)}^{2})$ and $\sharp V_{1,2}=O(d_{(n)})$,
\[
{\rm P}(W_{2,k}=0) \geq  1-\frac{Cd_{(n)}^{2}}{c_{n}}.
\]
By the same token,
\begin{equation*}
    {\rm P}(W_{x,2,k}=m) \geq 1-\frac{Cd_{(n)}^{2}}{c_{n}}.
\end{equation*}
This completes the proof.\end{proof}
\begin{proof}[Proof of $(\ref{Claim1})$] By Lemma \ref{LB11},
\begin{align}\label{SUB1}
     |{\rm P}
(W_{2,k}=0)-{\rm P}
(W_{x,2,k}=m )| \leq \frac{Cd_{(n)}^{2}}{c_{n}}.
\end{align}
Therefore, substituting (\ref{SUB1}) into (\ref{SUB2}), we have
\begin{align*}
   & |{\rm P}(W=k)- {\rm P}(W=k+m |  X_{A_{ijk}}=x )|\notag\\
   &\quad\leq {\rm P}(W_{1}=k)
 \frac{Cd_{(n)}^{2}}{c_{n}}+\bigg|\sum_{i=1}^{k}{\rm P}(W_{1}=k-i)[{\rm P}
(W_{2,k-i}=i)-{\rm P}
(W_{x,2,k-i}=m+i)]\bigg|.\notag
\end{align*}
Summing over $k$ from zero to $m_{n}$ yields
\begin{align}\label{C33}
    & \sum_{k=0}^{m_{n}}|{\rm P}(W=k)- {\rm P}(W=k+m | X_{A_{ijk}}=x)|\notag\\
    &\quad\leq  \frac{Cd_{(n)}^{2}}{c_{n}}+\bigg|\sum_{k=0}^{m_{n}}\sum_{i=1}^{k}{\rm P}(W_{1}=k-i) [{\rm P}(W_{2,k-i}=i)-{\rm P}(W_{x,2,k-i}=m+i) ]\bigg|\notag\\
&\quad\leq  \frac{Cd_{(n)}^{2}}{c_{n}}+\bigg|\sum_{k=0}^{m_{n}}\sum_{i=1}^{k}{\rm P}(W_{1}=k) [{\rm P}
(W_{2,k}=i)-{\rm P}
(W_{x,2,k}=m+i) ]\bigg|\notag\\
  &\quad\leq \frac{Cd_{(n)}^{2}}{c_{n}}+\sum_{k=0}^{m_{n}}{\rm P}(W_{2,k}> 0){\rm P}(W_{1}=k)+\sum_{k=0}^{\infty}{\rm P}(W_{x,2,k}> m){\rm P}(W_{1}=k)\notag\\
   &\quad= C\bigg\{\frac{d_{(n)}^{2}}{c_{n}}+2\sum_{k=0}^{m_{n}} \frac{d_{(n)}^{2}}{c_{n}}{\rm P}(W_{1}=k)\bigg\}\notag\\
   &\quad=  \frac{Cd_{(n)}^{2}}{c_{n}},
\end{align}
where in the third inequality we used Lemma \ref{LB11} again. Since $x$ is arbitrary, substituting (\ref{C33}) into~(\ref{C32}) proves that  $ S(W) \leq S_{2}(W)+Cd_{(n)}^{2}/c_{n}.$
\end{proof}

\section{Proof of (\ref{318})}\label{BB}
Fix $X_{A_{ijk}}=x_{A_{ijk}}$. Let $\mathcal{G}$ be all possible graphs with  the vertex set $V=\{1,2,\ldots, n\}$.  Denote by $V_{1}$   the vertices in $T_{i}$, $T_{j}$, and $T_{k}$, and set  $V_{2}=V\setminus V_1$.  Construct a stationary reversible Markov chain $\{Z_{n},n\geq 0\}$ on $\mathcal{G}$ as  follows.

\textbf{Step 1.}   The initial graph $Z_0$ consists of the following edges. Suppose that we are given a pair of vertices $i, j\in V$. If either $i$ or $j$ is from $V_1$, then assign an edge $e_{ij}$ as in $G(n, p)$; otherwise, generate a new edge with probability $p$ independently.

\textbf{Step 2.} Given the present graph $Z_n$,   $Z_{n+1}$ is obtained by  choosing two vertices from $Z_{n}$ uniformly at random, except from $A_{ijk}$, and independently resampling the edge between them.

It is easy to verify that the Markov chain constructed above is stationary and reversible.   Let $W_0, W_1$, and $W_2$  be the number of triangles in
$Z_0, Z_1$, and  $Z_2$, respectively. Define
    $$ Q_{m}(x)={\rm P} [W_1=W_0+m | Z_0=x ],\quad  q_{m}={\rm E}Q_{m}(Z_0)={\rm P}(W_1=W_0+m),$$
    and
    $$Q_{m_{1}, m_{2}}(x)={\rm P} [W_1=W_0+m_{1}, W_2=W_1+m_{2} | Z_0=x ].$$
We need the following lemma, which is a direct consequence
 of \cite[Theorem 3.7]{rollin2015local}.
\begin{lemma}\label{LB1}  For  any positive integers $m$, $m_{1}$, and $ m_{2}$, we have
\begin{align*}
        S_{2}(W_0) &\leq  \frac{1}{q_{m}^{2}}  (2 \operatorname{Var} Q_{m}(Z_0)+{\rm E} |Q_{m, m}(Z_0)-Q_{m}(Z_0)^{2} | \\
        &  \,\quad+2 \operatorname{Var} Q_{-m}(Z_0)+{\rm E} |Q_{-m,-m}(Z_0)-Q_{-m}(Z_0)^{2} |).
    \end{align*}
\end{lemma}
Next, it is sufficient to compute $\operatorname{Var} Q_{m}(Z_0)$ and ${\rm E}\left|Q_{m, m}(Z_0)-Q_{m}(Z_0)^{2}\right|$ for $m=1,-1$.
\begin{lemma}\label{LB2}
   We have
     \begin{align}\label{B1}
         &\operatorname{Var} Q_{1}(Z_0)=O(p^{5}),\\
         &\label{B2}
          \operatorname{Var} Q_{-1}(Z_0)=O\bigg(\frac{p^{3}}{n}\bigg).
     \end{align}
\end{lemma}
\begin{proof}  We follow the line of the proof in \cite{rollin2015local}. Denote by
    $E_{i,j}$  an indicator that there exists an edge between vertices $i$ and $j$, and set  $V_{i}^{k, j}=E_{i,j} E_{i,k}$. Define
      $$Y_{i}^{j, k}= (1-E_{j,k} ) V_{i}^{j, k} \prod_{l \neq i, j, k} (1-V_{l}^{j, k} )$$
      and
      $$p_{jk}=\left\{\!\begin{array}{ll}
        p  &\ \text{ if }j\text{ and }k \text{ are both in } V_{2}, \\
        0  &\ \text{ if }j\text{ or }k \text{ is in } V_{1}.
     \end{array}\right.$$
 Then,  it follows that
    \begin{equation}\label{Q11}
        Q_{1}(Z_0)=\frac{1}{\binom n2} \sum_{\{j, k\}} p_{jk}\sum_{i \neq j, k} Y_{i}^{j, k}.
    \end{equation}
Let $E_{1}$ be the set of all triples $\{r,s,t\}$ with  $s$ or $t$  in $V_{1}$,    $E_{2}$ be the set of all  triples $\{r,s,t\}$ with only $r$ in $V_{1}$, and $E_{3}$ be the set of all triples $\{r,s,t\}$ with all in $V_{2}$.  After some basic calculations,
 \begin{equation*}\label{B4}
     p_{st}{\rm E} Y_{r}^{s, t}=
     \left\{\begin{array}{ll}
       0  & \ \text{ if  } \{r,s,t\} \in E_{1},\\
       0\ \text{or}\ p(1-p) (1-p^{2} )^{n-7}  &\  \text{ if  } \{r,s,t\} \in E_{2},\\
        0\ \text{or}\ p^{3}(1-p) (1-p^{2} )^{n-8}   &\ \text{ if  } \{r,s,t\} \in E_{3}.\\
     \end{array}\right.
 \end{equation*}
For $\operatorname{Var}Q_{1}(Z_0)$, we need to calculate the  covariances $\operatorname{Cov} (Y_{r}^{s, t}, Y_{w}^{u, v} )$ since $Q_{1}$ is a  sum of $Y$'s.  It follows from   (\ref{Q11}) that
 \begin{align*}
     \operatorname{Var}Q_{1}(Z_0)=\bigg(\sum_{\substack{\{r,s,t\} \in E_{3}\\[1pt]\{u,v,w\} \in E_{3}}}+\sum_{\substack{\{r,s,t\} \in E_{2}\\[1pt]\{u,v,w\} \in E_{2}}}+2\sum_{\substack{\{r,s,t\} \in E_{2}\\[1pt]\{u,v,w\} \in E_{3}}}\bigg)p_{s,t}p_{u,v}\operatorname{Cov} (Y_{r}^{s, t}, Y_{w}^{u, v} ).
 \end{align*}
 We deal with each sum separately.

\textbf{Case 1.} $\{r,s,t\}$ and $\{u,v,w\}$ are both in $E_{3}$.

It follows directly from \cite[Lemma 4.12]{rollin2015local} that
$$\sum_{\substack{\{r,s,t\} \in E_{3}\\[1pt]\{u,v,w\} \in E_{3}}} p_{st}p_{uv}\operatorname{Cov} (Y_{r}^{s, t}, Y_{w}^{u, v} )=  O(n^{4}p^{5}). $$

\textbf{Case 2.}  $\{r,s,t\}$ and $\{u,v,w\}$ are both in $E_{2}$.

If $\{r,s,t\} = \{u,v,w\}$, then $$p_{st}p_{uv}\operatorname{Cov} (Y_{r}^{s, t}, Y_{w}^{u, v} )\leq p^{2}(1-p) (1-p^{2} )^{n-7} (1-(1-p) (1-p^{2} )^{n-7} ),$$
 which contributes $O(n^{2})$ equal covariance terms.

If   $r=w$, then
$$p_{st}p_{uv}\operatorname{Cov} (Y_{r}^{s, t}, Y_{w}^{u, v} )\leq p^{2}(1-p)^{2} (1-p^{2} )^{2n-18}(4p^{3}-7p^{4}+4p^{6}-p^{8}),$$ which
contribute $O(n^{4})$ equal covariance terms.

If $\sharp\{r,s,t\}\cap\{u,v,w\} =1$ and $\sharp \{s,t\}\cap\{u,v\} =1$, then
$$p_{st}p_{uv}\operatorname{Cov} (Y_{r}^{s, t}, Y_{w}^{u, v} )\leq p^{2}(1-p)^{2} ((1-p)^{n-8} (1+p-p^{2} )^{n-8}- (1-p^{2} )^{2 n-14} ),$$ which
contribute $O(n^{3})$ equal covariance terms.

If $r = w$ and $\sharp \{r,s,t\}\cap\{w,u,v\} =2$, then
$$p_{st}p_{uv}\operatorname{Cov} (Y_{r}^{s, t}, Y_{w}^{u, v} )\leq p^{2}(1-p)^{2}((1-2p^{2} + p^{3})^{n-8} - (1 - p^{2})^{2n-14}),$$
which contributes $O(n^{3})$ equal covariance terms.

If $r\neq w$ and $ \{s,t\}=\{u,v\} $, then
$$p_{st}p_{uv}\operatorname{Cov} (Y_{r}^{s, t}, Y_{w}^{u, v} )=0.$$

If $\sharp \{r,s,t\}\cap \{u,v,w\} =0$, then
$$p_{st}p_{uv}\operatorname{Cov} (Y_{r}^{s, t}, Y_{w}^{u, v} )\leq p^{2}(1-p)^{2} (1-p^{2} )^{2n-16}(4p^{3}-7p^{4}+4p^{6}-p^{8}),$$
which contributes $O(n^{4})$ equal covariance terms.

In other cases,
$$p_{st}p_{uv}\operatorname{Cov} (Y_{r}^{s, t}, Y_{w}^{u, v} )=0.$$

\textbf{Case 3.}   $\{r,s,t\}$ is in $E_{2}$ and $\{u,v,w\}$ is in $E_{3}$.

If $\sharp \{r,s,t\}\cap \{u,v,w\} =0$, then
$$p_{st}p_{uv}\operatorname{Cov} (Y_{r}^{s, t}, Y_{w}^{u, v} )\leq p^{4}(1-p)^{2} (1-p^{2} )^{2n-18}(4p^{3}-7p^{4}+4p^{6}-p^{8}),$$
which contributes $O(n^{5})$ equal covariance terms.

If   $s=w$ or $t=w$, then
$$p_{st}p_{uv}\operatorname{Cov} (Y_{r}^{s, t}, Y_{w}^{u, v} )\leq p^{4}(1-p)^{2} (1-p^{2} )^{2 n-17}(-2p+4p^{2}-3p^{4}+p^{6}),$$
which contributes $O(n^{4})$ equal covariance terms.

If  $\sharp\{r,s,t\}\cap\{u,v,w\} =1$, $\sharp \{s,t\}\cap\{u,v\} =1$, and $\sharp\{s,t\}\cap\{u,v\}=1$, then
$$p_{st}p_{uv}\operatorname{Cov} (Y_{r}^{s, t}, Y_{w}^{u, v} )\leq p^{4}(1-p)^{2} ((1-p)^{n-8} (1+p-p^{2} )^{n-9}- (1-p^{2} )^{2 n-15} ),$$
which contributes $O(n^{4})$ equal covariance terms.

If $r\neq w$ and $ \{s,t\}=\{u,v\} $, then
$$p_{st}p_{uv}\operatorname{Cov} (Y_{r}^{s, t}, Y_{w}^{u, v} )=0.$$

In other cases,
$$p_{st}p_{uv}\operatorname{Cov} (Y_{r}^{s, t}, Y_{w}^{u, v} )=0.$$

Combining   Cases  1--3, we prove  (\ref{B1}).

A similar argument can establish (\ref{B2}) since $Q_{-1}(Z_0)$ can be written as
$$Q_{-1}(Z_0)=\frac{1}{\binom n2} \sum_{\{j, k\}} \sum_{i \neq j, k}(1-p_{jk}) E_{j, k} V_{i}^{j, k} \prod_{l \neq i, j, k} (1-V_{l}^{j, k} ).$$
The proof is complete.
\end{proof}
\begin{lemma}\label{LB3}
  We have
    \begin{align*}\label{B3}
        &{\rm E}|Q_{1,1}(Z_0)-Q_{1}(Z_0)^{2}|\leq  \frac{2p^{4}(1-p)(1-p^{2})^{n-8}}{n-1},\\
        &{\rm E}|Q_{-1,-1}(Z_0)-Q_{-1}(Z_0)^{2}|\leq \frac{2p^{3}(1-p)^{2}(1-p^{2})^{n-8}}{n-1}.
    \end{align*}
\end{lemma}
\begin{proof}
It easily follows that
     $$Q_{1,1}(Z_0)= \sum_{\{s, t\}} \frac{p_{st}}{\binom n2} \sum_{r \neq s, t} Y_{r}^{s, t} \sum_{\{u, v\} \neq\{s, t\}}\frac{p_{uv}}{\binom n2} \sum_{w \neq u, v} Y_{w}^{u, v}.$$
  For fixed $\{j,k\}$, since at most one $i\in V$ such that $\{Y_{i}^{jk}\}$ is possibly non-zero,
     \begin{equation}\label{B8}
         {\rm E} |Q_{1,1}(Z_0)-Q_{1}(Z_0)^{2} |={\rm E} \sum_{\{s, t\}}\frac{p_{st}^{2}}{\binom n2 ^{2}} \sum_{r \neq s, t} Y_{r}^{s, t}\leq\frac{p}{\binom n2 } q_{1} ,
     \end{equation}
    By the same token, we get
    \begin{equation}\label{B9}
        {\rm E} |Q_{-1,-1}(Z_0)-Q_{-1}(Z_0)^{2} |\leq \frac{1-p}{\binom n2 } q_{1}.
    \end{equation}
Simple calculations yield
\begin{equation}\label{B10}
 q_{1}={\rm E} Q_{1}(Z_0)\leq (n-2) p^{3}(1-p) (1-p^{2} )^{n-8}.
\end{equation}
Substituting (\ref{B10}) into (\ref{B8}) and (\ref{B9}), we see that Lemma \ref{LB3} holds.
\end{proof}

\end{appendix}


\begin{thebibliography}{99}



\bibitem{arratia1989two} Arratia R, Goldstein L, Gordon L.
 Two moments suffice for Poisson approximations: The Chen-Stein method. Ann Probab, 1989, 17: 9--25

\bibitem{barbour2002total} Barbour A D, \'{C}ekanavi\'{c}ius V.
Total variation asymptotics for sums of independent integer random variables. Ann Probab, 2002, 30: 509--545

\bibitem{barbour1992poisson} Barbour A D, Holst L, Janson S.
 Poisson Approximation.
  New York: Oxford Univ  Press, 1992

\bibitem{barbour1989central}
 Barbour A D, Karo\'{n}ski M, Ruci\'{n}ski A.
 A central limit theorem for decomposable random variables with applications to random graphs. \href{https://doi.org/10.1016/0095-8956(89)90014-2}{J Combin Theory Ser B}, 1989, 47: 125--145

\bibitem{barbour1999poisson} Barbour A D, Xia A.
Poisson perturbations.
\href{https://doi.org/10.1051/ps:1999106}{ESAIM Probab Stat}, 1999, 3: 131--150

\bibitem{brown2001stein} Brown T C, Xia A.
Stein's method and birth-death processes. \href{https://doi.org/10.1214/aop/1015345606}{Ann Probab}, 2001, 29: 1373--1403

\bibitem{chen2011normal} Chen L H Y, Goldstein L, Shao Q M.
 Normal Approximation by Stein's Method.
 Berlin: Springer, 2010

\bibitem{fang2015universal} Fang X.
A universal error bound in the CLT for counting monochromatic edges in uniformly colored graphs.
Electron Commun Probab, 2015, 20: 1--6
%,pp 21

\bibitem{fang2019wasserstein} Fang X.
Wasserstein-$2$ bounds in normal approximation under local dependence. Electron J Probab, 2019, 24: 35
%, 14pp

\bibitem{goovaerts1996compound} Goovaerts M J, Dhaene J.
 The compound Poisson approximation for a portfolio of dependent risks. \href{https://doi.org/10.1016/0167-6687(95)00033-X}{Insurance Math Econom}, 1996, 18: 81--85

\bibitem{liu2022wasserstein} Liu T, Austern M.
Wasserstein-$p$ bounds in the central limit theorem under weak dependence. Electron J Probab, 2023, 28: 117
%, 47pp

\bibitem{rollin2008symmetric} R\"{o}llin A.
Symmetric and centered binomial approximation of sums of locally dependent random variables.
Electron J Probab, 2008, 13: 756--776

\bibitem{rollin2022kolmogorov} R\"{o}llin A.
Kolmogorov bounds for the normal approximation of the number of triangles in the Erd\H{o}s-R\'{e}nyi random graph. \href{https://doi.org/10.1017/S0269964821000061}{Probab Engrg Inform Sci}, 2022, 36: 747--773

\bibitem{rollin2015local} R\"{o}llin A, Ross N.
 Local limit theorems via Landau-Kolmogorov inequalities.
  Bernoulli, 2015, 21: 851--880

\bibitem{ross2011fundamentals} Ross N. Fundamentals of Stein's method. Probab Surv, 2011, 8: 210--293

\bibitem{shao2019berry} Shao Q M, Zhang Z S.
Berry-Esseen bounds of normal and nonnormal approximation for unbounded exchangeable pairs.
 Ann Probab, 2019, 47: 61--108

\bibitem{smith1988extreme} Smith R L.
 Extreme value theory for dependent sequences via the Stein-Chen method of Poisson approximation. \href{https://doi.org/10.1016/0304-4149(88)90092-0}{Stochastic Process Appl}, 1988, 30: 317--327

\bibitem{su2022approximation} Su Z, Ulyanov V V, Wang X.
 Approximation of sums of locally dependent random variables via perturbation of Stein operator. Theory Probab Appl,  in press, 2025

\bibitem{upadhye2017Stein} Upadhye N S, \'{C}ekanavi\'{c}ius V, Vellaisamy P.
 On Stein operators for discrete approximations.
  Bernoulli, 2017: 23: 2828--2859

\bibitem{vatutin1982limit} Vatutin V A, Mikhailov V G.
 Limit theorems for the number of empty cells in the equiprobable scheme of group disposal of particles.
 Teor Veroyatnost I Primenen, 1982, 27: 684--692

\bibitem{wainer1996reliability} Wainer H, Thissen D.
 How is reliability related to the quality of test scores? What is the effect of local dependence on reliability? Educ Meas Issues Pract, 1996, 15: 22--29







\end{thebibliography}
\end{document}